\documentclass[11pt,reqno]{amsart}

\usepackage[T1]{fontenc}
\usepackage{graphicx}
\usepackage{cite,enumerate}
\usepackage{multicol}

\usepackage{color}
\definecolor{MyLinkColor}{rgb}{0,0,0.4}

\newtheorem{thm}{Theorem}[section]
\newtheorem{prop}[thm]{Proposition}
\newtheorem{lemma}[thm]{Lemma}
\newtheorem{cor}[thm]{Corollary}
\newtheorem{defn}[thm]{Definition}

\theoremstyle{remark} 
\newtheorem{rem}[thm]{Remark}

\numberwithin{equation}{section}

\title[On the parabolicity of the Muskat problem]{On the parabolicity of the Muskat problem: Well-posedness, fingering,  and  stability results}

\author[Escher]{Joachim Escher}
\address{Institut f{\"u}r Angewandte Mathematik, Leibniz Universit{\"a}t Hannover, Welfengarten~1, 30167 Hannover, Germany. }
\email{escher@ifam.uni-hannover.de}

\author[Matioc]{Bogdan--Vasile Matioc}
\address{Institut f{\"u}r Angewandte Mathematik, Leibniz Universit{\"a}t Hannover, Welfengarten~1, 30167 Hannover, Germany. }
\email{matioc@ifam.uni-hannover.de}

\usepackage[colorlinks=true,linkcolor=MyLinkColor,citecolor=MyLinkColor]{hyperref} 

\begin{document}

\begin{abstract}
We consider in this paper the  Muskat problem in a periodic geometry and incorporate  capillary as well as gravity effects in the modelling.
The problem is rewritten as an abstract evolution equation   and we use this property to prove  well-posedness of the problem and to
 establish  exponential stability  of some flat equilibrium.
 Using bifurcation theory we also find   finger shaped steady-states  which are all unstable.
\end{abstract}

\maketitle

\section{Introduction}
The Muskat problem was introduced into the mathematical literature by  M. Muskat \cite{M}
as a model for the simultaneous flow of  two immiscible fluids in a porous medium.
The problem is of considerable practical relevance since it is also a model for the  interaction between oil and water in an oil reservoir.
In practice high flow water injection pumps fill the bottom of the reservoir with  water to push the oil towards the wells like a piston.
Furthermore the injected water  is used to keep the  pressure  in the reservoir unchanged over a long term, which is  keeping the production rate constant.
From the practical point of view it may be of interest to determine the optimal speed at which  water drives the oil, which adheres more strongly to the  porous structure, upwards to the oil well.

In contrast to the  Mullins-Sekerka problem, where the interface separating the  two fluids is driven forward with normal velocity equal to the
normal component of the  velocity jump over the interface, cf. \cite{ES98}, the Muskat problem considers a continuous velocity field 
in the whole domain occupied by the two fluids.
The authors  in \cite{ES98} show  that the Mullins-Sekerka problem has a very natural parabolic structure
and solve it by using the  rich literature on this  topic.
A first result on the classical solvability of the  Muskat problem was
obtain in \cite{Y1}, where  local in time existence of classical solutions is established by using Newton's iteration method. 
Later on,  in \cite{Yi2}, the Schauder fixed point theorem  and existence results for a elliptic diffraction problem are used to prove 
existence of global solutions for small perturbations of some flat interface.  

The problem has been also considered by means of complex methods.
We like to mention the work \cite{SCH} where global in time existence  for initial data
that are  small perturbations of a flat interface are proved.
The authors of \cite{SCH} also show that the problem is ill-posed when the less-viscous fluid drives forward the fluid with higher viscosity.
Further local well-posedness results have been established \cite{Ambr} and more recently in \cite{CCG1, CCG2} by means of energy estimates in the absence of surface tension forces.

 Regarding the stability of equilibria an interesting result is found in \cite{FT} where  
the stability of some circular steady-state is proven by considering series expansions and surface tension effects at the boundary between the fluids. 
The authors of \cite{FT} state that  the  equilibrium  is not generally asymptotically stable.
As far as we know there are no other stability results for the  Muskat problem. 
Both, the Mullins-Sekerka and the Muskat problem are also closely related to fingering phenomena which have and still receive a lot of interest in many fields of  sciences.

On  the parabolicity of the Muskat problem not many results are available.
We show in this paper that the Muskat problem and the Mullins-Sekerka model are related in the sense that
the Muskat problem is, at least in a neighbourhood of some flat interface, of parabolic type too.
This observation is true, when considering surface tension effects, independently of the boundary data.
When allowing only for  gravity and viscosity
effects we need to make  restrictions to the  boundary data which are consistent with earlier results  mentioned above.
The approach proposed here  is not only a useful tool to prove the well-posedness of the problem but enables us also to use the  theory on parabolic problems to 
establish the dependence of the solutions on the initial data (see Theorem \ref{T:main}).
Furthermore, parabolic theory provides us a simple argument to prove  exponential stability of some flat equilibrium, cf. Theorem \ref{T:main2}.
Additionally,  Corollary \ref{C:1} yields an optimal value for the normal velocity at which water may displace  oil in a oil reservoir.
When considering surface tension effects, we find in Theorem \ref{T:P3}  for small values of the surface tension coefficient steady-state fingering solutions which are all unstable, cf. Theorem \ref{T:instability}. 
In the Appendix we show that the problem under certain constant boundary conditions corresponds to a  Muskat problem in a frame moving with constant velocity.

The outline of the paper is as follows. 
We start with a short description of the model of interest and present the well-posedness result in Section 2.
In order to establish this result, we transform in Section 3 the problem on a fixed domain by straightening the  unknown boundary 
and reduce the problem to an abstract evolution equation.
In Section \ref{S:6} we prove the asymptotic stability of some flat equilibrium which appears under certain constant boundary conditions.
In the following section
we find, using the surface tension coefficient as a bifurcation parameter, global bifurcation branches consisting only of unstable steady-state solutions of the problem.
In the Appendix we analyse the Muskat problem in a moving frame.

\section{Mathematical model and the well-posedness result}
To give a precise description of the physical setting we are interested in, consider a porous medium (or a vertical Hele-Shaw cell)
containing two incompressible and immiscible Newtonian fluids  in motion.
We use the subscript $-$ when referring to the fluid located in the lower part of the porous medium and which  occupies the domain $\Omega_-(t),$ 
respectively $+$ for the other one located in $\Omega_+(t)$.
These two fluid phases are separated by the interface $\Gamma(t)$ which evolves in time and is to be determined as a part of the problem.
The physically relevant problem is essentially three-dimensional, but it may be  approximated by a 
two-dimensional mathematical model.
We refer to \cite{KPS} for a deduction of a generalised Darcy's law for non-Newtonian fluids.
The methods presented there apply  also to Newtonian fluids and yield for a fluid in motion in a Hele-Shaw cell  the well-known  linear Darcy's law for a two-dimensional averaged problem.
\begin{figure}
\hspace{0.5cm}$$\includegraphics[scale=0.35]{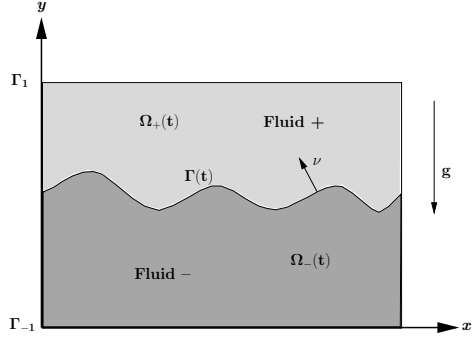}$$
\caption{The 2-D setting}
\end{figure}
This is the reason why we restrict our considerations to the situation $\Omega_\pm(t)\subset\mathbb{R}^2.$
The Muskat problem is a potential flow in the sense that the 
velocity fields $\vec v_\pm$ 
satisfy Darcy's law
\begin{subequations}\label{eq:P3}
\begin{equation}\label{eq:P3-1}
\vec  v_\pm=-\frac{k}{\mu_\pm}\nabla u_\pm \quad\text{in}\quad\Omega_\pm(t),
\end{equation}
cf. \cite{M}.
The potentials 
\begin{equation}\label{potentials}
u_\pm:=p_\pm +g\rho_\pm y\quad\text{in}\quad\Omega_\pm(t)
\end{equation}
incorporate both the dynamic pressure $p_\pm$ inside $\Omega_\pm(t)$ as well the hydrostatic term $g\rho_\pm y.$
 We  write $\rho_\pm$ to denote  the densities
of the fluids, $g$ is  the gravity constant, $k$ the permeability of the porous medium,
$\mu_\pm$ are the viscosity constants,  and we have chosen $y$ to be the height coordinate (see Figure 1.).
Incompressibility means that
\begin{equation}\label{eq:P3-2}
\nabla\cdot\vec v_\pm=0\quad\text{in}\quad\Omega_\pm(t).
\end{equation}
The pressure jump across the interface $\Gamma(t)$ is compensated by the surface tension in the interface.
From the Laplace-Young condition   we then obtain
\begin{equation}\label{eq:P3-3}
u_+-u_-=\gamma\kappa_{\Gamma(t)}+g(\rho_+-\rho_-)y\quad\text{on}\quad\Gamma(t),
\end{equation}
where   $\gamma$ is the surface tension 
coefficient of the interface, and   $\kappa_{\Gamma(t)}$ is the curvature of $\Gamma(t)$.
Two more, so-called kinematic boundary conditions on $\Gamma(t)$ are given (see also \cite{FT, M, Y1, Yi2}) by
 \begin{equation}\label{eq:P3-6}
V(t,\cdot)=\langle\vec v_\pm(t,\cdot),\nu(t,\cdot)\rangle\quad\text{on}\quad \Gamma(t),
\end{equation}
meaning that the interface moves along with the fluids (a particle on the boundary remains on the boundary as time elapses).
Here,  $\nu(t)$ stands for the  unit normal at $\Gamma(t)$ which points into $\Omega_+(t).$
In order to study the problem we are left to impose boundary conditions on the fixed interfaces $\Gamma_{-1}$ and $\Gamma_1.$
On the bottom $\Gamma_{-1}$ we presuppose the value of $u_-$ to be known, 
\begin{equation}\label{eq:P3-4}
 u_-=g_1(t,x)\quad\text{on}\quad \Gamma_{-1},
\end{equation}
and  on $\Gamma_1$ we set
\begin{equation}\label{eq:P3-5}
\partial_{\overline\nu} u_+=g_2(t,x)\quad\text{on}\quad \Gamma_1,
\end{equation}
where $\overline\nu=(0,1) $ is   the unit  normal vector at $\Gamma_1$ which points outward from $\Omega_+(t).$
Of course, it is also possible to choose Dirichlet conditions on $\Gamma_+$ and Neumann conditions on $\Gamma_{-1},$
 but we prefer to remain within the setting considered in  \cite{Y1, Yi2}.
The interface at time $t=0$ is also  prescribed:
\begin{equation}\label{eq:P3-7}
\Gamma(0)=\Gamma_0.
\end{equation}
\end{subequations}

System \eqref{eq:P3} is a two-phase moving boundary problem.
The main interest is in determining the motion of the interface $\Gamma(t)$ separating the fluids. 
If $\Gamma(t)$ is known, then the potentials $u_\pm$ can be then found by solving elliptic mixed boundary value problems. 
We shall prove that if the boundary $\Gamma_0$ is the graph of a  function which is small in some suitable Banach space, then, locally in time,
problem  \eqref{eq:P3} possesses a unique classical H\"older solution.
To that aim, we  first rewrite problem \eqref{eq:P3} in a more accessible way by introducing a parametrisation 
of the, a priori unknown, moving boundary $\Gamma(t).$

\subsection{Parametrising the boundary}

For simplicity, we consider in the following horizontally  $2\pi-$periodic flows only. 
Let $\alpha\in(0,1)$ be fixed in the remainder of this work. 
We define the set of admissible functions to be
\[
\mathcal{V}:=\{f\in h^{2+2{\mathop{\rm sign}\,}(\gamma)+\alpha}(\mathbb{S})\,:\, \|f\|_{C(\mathbb{S})}<1\},
\]
where ${\mathbb{S}}$ is the unit circle.
Note that when we neglect surface tension $\mathcal{V}\subset h^{2+\alpha}(\mathbb{S}),$ whereas for $\gamma>0$ we have $\mathcal{V}\subset h^{4+\alpha}(\mathbb{S}).$
The small H\"older space  $h^{m+\beta}(\mathbb{S})$, $m\in\mathbb{N}$ and $\beta\in(0,1)$, is defined as  the completion  of $C^\infty(\mathbb{S})$
in $C^{m+\beta}(\mathbb{S}),$ the classical H\"older spaces.
Furthermore, we set  $\Gamma_k:=\mathbb{S}\times\{k\}$ for $k\in\{-1,0,1\}$.

To incorporate time, we fix $T>0$. 
If  $f\in C([0,T], \mathcal{V})\cap C^1([0,T], h^{1+\alpha}(\mathbb{S})) $ describes the evolution of the interface 
separating fluids, then, at any time $t\in[0,T],$ we have that  
 $\Omega_\pm(t)=\Omega_\pm(f(t)),$ where, given $h\in\mathcal{V},$ we set
\begin{align*}
&\Omega_-(h):=\{(x,y)\in\Omega\,:\, -1<y<h(x)\}, \\[1ex]
& \Omega_+(h):=\{(x,y)\in\Omega\,:\, h(x)<y<1\}.
\end{align*}
Consequently, it holds that $\Gamma(t)=\Gamma(f(t))$ for all $t\in[0,T],$ with
\begin{align*}
&\Gamma(h):=\{(x,h(x))\,:\, x\in\mathbb{S}\},
\end{align*}
for all $h\in\mathcal{V}.$
Let us describe   the kinematic condition \eqref{eq:P3-6} in this particular context.
Consider the evolution of a  particle $(x(t),y(t))$ on the interface $\Gamma(f(t)).$ 
Since particles on the interface separating the fluids remain there as time evolves, we obtain from $\Gamma(t)=\Gamma(f(t))$, 
 that $y(t)=f(t,x(t))$ for all $t\in[0,T].$
Differentiating  this equation with respect to the time variable $t$ yields 
$\partial_t f=(-f',1)(x',y'),$ hence, the normal velocity $V$ satisfies the relation
 \[
V=\frac{\partial_tf}{\sqrt{1+f'^2}}.
\]
We denote  herein by $f'$ the spatial derivative of $f$.
Thereby, the local study of problem  \eqref{eq:P3} reduces to the following system
\begin{equation}\label{eq:S}
\left\{
\begin{array}{rllllll}
\Delta u_\pm&=&0&\text{in}& \Omega_\pm(f(t)), &0\leq t\leq T,\\[1ex]
\partial_y u_+&=&g_2&\text{on}& \Gamma_1,& 0\leq t\leq T,\\[1ex]
u_-&=&g_1&\text{on}&\Gamma_{-1},&0\leq t\leq T,\\[1ex]
u_+-u_-&=&\gamma\kappa_{\Gamma(f)}+\varpi  f&\text{on}& \Gamma(f(t)),&0\leq t\leq T,\\[1ex]
\partial_tf+k\mu_\pm^{-1}\sqrt{1+f'^2}\partial_\nu u_\pm&=&0 &\text{on}& \Gamma(f(t)),&0< t\leq T,\\[1ex]
f(0)&=&f_0, &
\end{array}
\right. 
\end{equation} 
where $f_0\in\mathcal{V}$ determines the initial shape of $\Gamma_0$, and to simplify notation, we set
\begin{equation}\label{eq:notation}
\varpi:=g(\rho_+-\rho_-). 
\end{equation}
Vice versa, defining $\vec v_\pm$ by \eqref{eq:P3-1},
  system \eqref{eq:P3} can be deduced from \eqref{eq:S}.
The functions $g_1$ and $g_2$ are assumed to satisfy
\[
\text{$g_1\in C([0,\infty), h^{2+\alpha}(\mathbb{S})) $ and $g_2\in C([0,\infty), h^{1+\alpha}(\mathbb{S})) $.}
\]

 A triple $(f,u_+,u_-) $ is called  a {\em classical H\"older solution of \eqref{eq:S}} on $[0,T]$ if 
\begin{align*}
&f\in C([0,T],\mathcal{V})\cap C^{1}([0,T],h^{1+\alpha}(\mathbb{S})),\\[1ex]
&u_\pm(t,\cdot)\in \mbox{\it buc}^{2+\alpha}(\Omega_\pm(f(t))),\quad t\in[0,T], 
\end{align*} 
and if $(f,u_+,u_-)$ satisfies \eqref{eq:S} pointwise. 
Given $f\in\mathcal{V},$
we let $\mbox{\it buc}^{2+\alpha}(\Omega_\pm(f))$ denote  the completion of the smooth functions $\mbox{\it BUC}\,^{\infty}(\Omega_\pm(f))$
in the Banach space $\mbox{\it BUC}\,^{2+\alpha}(\Omega_\pm(f))$.  
The first main result of this paper is the following theorem:
\begin{thm}[Local well-posedness]\label{T:main} Let $\gamma\in[0,\infty)$, $c_1, c_2\in\mathbb{R}$, and assume that
\begin{equation}\label{eq:cond}
\gamma>0\qquad\text{or}\qquad g(\rho_+-\rho_-)+c_2\left(\frac{\mu_-}{\mu_+}-1\right)<0.
\end{equation}
Then there exist open  neighbourhoods of the zero function $\mathcal{O}_i\subset h^{i+\alpha}(\mathbb{S})$, $i\in\{1,2\},$ and $\mathcal{O}\subset \mathcal{V}$,
 such that for all 
 $f_0\in \mathcal{O}$ and $g_i\in C([0,\infty), c_i+\mathcal{O}_i),$  $i=1,2,$
there exists $T(f_0)>0$ and a unique maximal H\"older   solution $(f=:f(\cdot;f_0),u_+,u_-)$ of problem \eqref{eq:S} on  $[0,T(f_0))$ which fulfills $f(t)\in\mathcal{O}$ for all $t\in[0,T(f_0)).$

If $g_1, g_2$ is of class $ C^{j}$,  then  
\[
\{(t,f_0)\,:\, \text{$f_0\in\mathcal{O}$ and $t\in(0,T(f_0))$}\}\to h^{2+2{\mathop{\rm sign}\,}(\gamma)+\alpha}(\mathbb{S}),\quad (t, f_0)\mapsto f(t; f_0)
\]
is of class $C^{j}.$

Furthermore, if $\gamma=0$, $g_1, g_2$ are constants and if
\begin{equation}\label{eq:cond2}
g(\rho_+-\rho_-)+g_2\left(\frac{\mu_-}{\mu_+}-1\right)>0,
\end{equation}
 then the linearised problem is ill-posed in the sense of Hadamard.
\end{thm}
 
 \begin{rem} If  $\gamma>0$ then we may choose $\mathcal{O}_i=h^{i+\alpha}(\mathbb{S})$, $i\in\{1,2\}.$
This means that we need no restriction on the boundary data when surface tension effects are considered.
 \end{rem}

Notice that if $g_1$ and $g_2$ depend only upon time and are constant in the spatial variable, then, 
for  $T>0$ small enough,  
\begin{equation}\label{eq:sim}
f(t):=-\frac{k}{\mu_+}\int_0^tg_2(s)\, ds,\qquad t\in[0,T],
\end{equation}
  is a solution of the problem (see also \cite{Yi2}).
 This is a flat interface which moves according to the sign of $g_2.$
Consider now the tuple  $(f,u_+,u_-) $ to be  an arbitrary solution of \eqref{eq:S}.
 Stokes'  theorem yields then 
 \begin{align*}
 \frac{\mu_+}{k}\frac{d}{dt}\int_\mathbb{S} f\,dx=-\int_\mathbb{S}\sqrt{1+f'^2}\partial_\nu u_+\, dx=-2\pi\int_{\Gamma_f}\partial_\nu u_+=-2\pi\int_\mathbb{S} g_2\, dx,
 \end{align*}
 which shows that the upper fluid moves towards the bottom if $g_2>0.$
If $g_2<0$  the fluid on the bottom must rise. 
This behaviour may be seen  very well also from  relation \eqref{eq:sim}.
Moreover, if $g_2$ has integral mean equal to zero, then the moving  interface 
oscillates around some constant value, since the formula above yields in this case
\[
\int_\mathbb{S} f\, dx=const.
\]
Summarising,  the direction of the flow is determined only by the boundary data $g_2$, while 
 $g_1$  is related only with the values of the potentials $u_\pm.$

\begin{rem} The result stated in Theorem \ref{T:main}  is consistent with previous results on this problem.
For example, if the densities of the fluids are equal, but the more viscous fluid drives forward the less viscous one
the problem is well-posed, cf. \cite{SCH, Y1, Yi2}.
Furthermore, if the  lower-viscosity fluid expands
into the higher-viscosity fluid, then the linearised problem is ill-posed, a result which is related with  the investigation in \cite{SCH}.
\end{rem}

Our analysis discloses that the problem is well-posed also when $\rho_-<\rho_+$ and $\mu_->\mu_+.$
In this case 
\[
c_2<\frac{g\mu_+(\rho_--\rho_+)}{\mu_--\mu_+}<0,
\]
thus the less dense fluid replaces the less viscous one, but the velocity of the flow must be sufficiently large.
Moreover, when $\rho_->\rho_+$ and $\mu_-<\mu_+,$ the problem is well-posed provided
\[
c_2>\frac{g\mu_+(\rho_--\rho_+)}{\mu_--\mu_+}.
\]
Notice that in this situation $c_2$ may be negative, which shows that the less viscous fluid may drive upwards the less dense one if the velocity of the system is small enough.
It is known that water is more dense than oil but  oil is more viscous  so that it adheres more strongly to the pores of the medium, cf. \cite{CC}.
Hence, this relation  determines also an optimal value of the normal velocity at which the water may replace the oil in an oil reservoir when surface tension effects are neglected.
Above this threshold value fingering may occur and water fingers could penetrate into the oil domain.
In view of \eqref{eq:P3-2} we obtain:

\begin{cor}\label{C:1} The optimal normal velocity at which oil is driven upwards by  water  in an oil reservoir, when surface tension effects are neglected, is
\begin{equation}\label{eq:optimal}
V_{opt}=\frac{gk(\rho_--\rho_+)}{\mu_+-\mu_-},
\end{equation}
if we identify the fluid on the bottom to be water and the one above as oil.
\end{cor}

Relation \eqref{eq:cond} was found also before,
 cf. \cite{ST}, when studying  the  Saffman-Taylor linearised stability of  the flat interface between two fluids in motion under small perturbations.

\section{The operator equation}
 
In order to prove Theorem  \ref{T:main} we first  transform the system \eqref{eq:S} into a problem on a fixed domain.
We  then use  certain solution operators for elliptic boundary value problems associated  to \eqref{eq:S} 
to rewrite the original problem as an abstract evolution equation.
Let us  introduce  some notation. 
We set $\Omega_\pm:=\Omega_\pm(0)$ 
and identify the boundary $\Gamma_0=\mathbb{S}\times\{0\}$  with the unit circle $\mathbb{S}$. 
For every $f\in\mathcal{V}$, we define the mappings $\phi_{ f\pm}=(\phi_{f\pm}^1,\phi_{f\pm}^2):\Omega_{\pm}\to\Omega_{\pm}(f)$
by 
\[
\phi_{f\pm}(x,y):=(x,y+(1\mp y)f(x)),\quad(x,y)\in\Omega_\pm.
\]
Note that   $\phi_{f\pm}$ are diffeomorphisms mapping the reference domains $\Omega_\pm$
 onto $\Omega_{\pm}(f),$ i.e. $\phi_{f\pm}\in\mbox{\it{Diff}}\,^{4+\alpha}(\Omega_\pm,\Omega_\pm(f)).$
 Moreover, $\phi_{f\pm}(\Gamma_{\pm1})=\Gamma_{\pm1},$
 $\phi_{f\pm}(\Gamma_0)=\Gamma(f)$, and the inverse mappings
 $\psi_{f\pm}$ are given by the expressions
\[
\quad\psi_{f\pm}(x,y):=\phi_{f\pm}^{-1}(x,y)=\left(x,\frac{y-f(x)}{1\mp f(x)}\right),\quad(x,y)\in\Omega_\pm(f).
\]

These diffeomorphisms induce  push-forward and pull-back operators 
\begin{eqnarray*}
&\phi_*^{f\pm}:\mbox{\it{buc}}\,^{2+\alpha}(\Omega_\pm)\to \mbox{\it{buc}}\,^{2+\alpha}(\Omega_\pm(f)), &v\mapsto v\circ\psi_{f\pm},\\[1ex]
&\phi_{f\pm}^*:\mbox{\it{buc}}\,^{2+\alpha}(\Omega_\pm(f))\to \mbox{\it{buc}}\,^{2+\alpha}(\Omega_\pm), &u\mapsto u\circ\phi_{f\pm},
\end{eqnarray*}
which are, in virtue of  the mean value theorem, isomorphisms, i.e. $(\phi_{f\pm}^*)^{-1}=\phi_*^{f\pm}.$
We now use  these  operators  to transform system 
\eqref{eq:S} on our fixed reference  domains $\Omega_\pm$.
The price to be paid for doing that is the fact that the differential operators we have to study are more involved.
Given $f\in\mathcal{V}$, we define the  operators $\mathcal{A}_\pm(f):\mbox{\it{buc}}\,^{2+\alpha}(\Omega_\pm)\to \mbox{\it{buc}}\,^{\alpha}(\Omega_\pm)$ by the relation
\[
 \mathcal{A}_\pm(f):=\phi_{f\pm}^*\Delta\phi_*^{f\pm}.
\]
The uniformly elliptic operators $\mathcal{A}_\pm(f)$ depend analytically on the a priori unknown function $f$, i.e.
\begin{equation}\label{eq10}
\mathcal{A}_\pm\in C^\omega(\mathcal{V}, \mathcal{L}(\mbox{\it{buc}}\,^{2+\alpha}(\Omega_\pm), \mbox{\it{buc}}\,^{\alpha}(\Omega_\pm))),
\end{equation} 
since we have the following representations of $\mathcal{A}_\pm$
\begin{eqnarray*}
\mathcal{A}_\pm(f)&=&\frac{\partial^2 }{\partial x^2}-2\frac{(1\mp y)f'}{1\mp f}\frac{\partial^2 }{\partial x\partial y}+
\left(\frac{(1\mp y)^2f'^2}{(1\mp f)^2}+\frac{1}{(1\mp f)^2}\right)\frac{\partial^2 }{\partial y^2}+\\[1ex]
&&-(1\mp y)\frac{(1\mp f)f''\pm 2f'^2}{(1\mp f)^2}\frac{\partial}{\partial y}\quad\text{for} \, f\in\mathcal{V}.
\end{eqnarray*}
Indeed, $\mathcal{A}_\pm(f)$ is a linear combination  of differential operators of order less or equal  to $2$
 with coefficients depending analytically on  
$f$.  
Furthermore, $\mathcal{A}_\pm(f)$ have the following geometric interpretation.
Let $\eta$ be  the standard metric on $\mathbb{R}^2.$ 
The diffeomorphisms $\phi_{f\pm}$ induce  Riemannian metrics, $\phi^*_{f\pm}\eta,$  
on $\Omega_\pm,$ i.e.
\[
\phi^*_{f\pm}\eta\big|_x(\xi,\zeta):=\eta\big|_{\phi_{f\pm}(x)}({\partial} \phi_{f\pm}(x)\xi,{\partial}\phi_{f\pm}(x)\zeta)
=\xi^T(\phi_{f\pm}(x))^T{\partial} \phi_{f\pm}(x)\zeta
\]
for $x\in\Omega_\pm$ and tangent vectors $\xi, \zeta\in T_x\Omega_\pm.$
With this notation, $\mathcal{A}_\pm(f)$ are exactly the Laplace-Beltrami operators corresponding to the Riemannian manifolds $(\Omega_\pm, \phi^*_{f\pm}\eta),$
respectively.
Furthermore, given $f\in\mathcal{V},$ we define the boundary operators $\mathcal{B}_\pm(f):\mbox{\it{buc}}\,^{2+\alpha}(\Omega_\pm)\to h^{1+\alpha}(\mathbb{S}),$
\begin{align*}
\mathcal{B}_\pm(f)v_\pm:=k\mu_\pm^{-1}{\mathop{\rm tr}\,} \phi_{f\pm}^*\langle \nabla(\phi_*^{f\pm}v_\pm)|(-f',1)\rangle,\quad v_\pm\in \mbox{\it{buc}}\,^{2+\alpha}(\Omega_\pm),
\end{align*}
where ${\mathop{\rm tr}\,}$ is the trace operator with respect to $\mathbb{S},$ i.e. ${\mathop{\rm tr}\,} v_\pm(x):=v_\pm(x,0)$ for all $x\in\mathbb{S}$
and $v_\pm\in\mbox{\it{BUC}}\,(\Omega_\pm).$
A simple computation shows 
\[
\mathcal{B}_\pm(f)v_\pm=\frac{k}{\mu_\pm}\left(\frac{1+f'^2}{1\mp f}{\mathop{\rm tr}\,} \frac{{\partial} v_\pm}{\partial y}-f'{\mathop{\rm tr}\,} \frac{{\partial} v_\pm}{\partial x}\right),
\]
so that $\mathcal{B}_\pm$ depend analytically on $f$, i.e.
\begin{equation}\label{anaB}
\mathcal{B}_\pm\in C^\omega(\mathcal{V}, \mathcal{L}(\mbox{\it{buc}}\,^{2+\alpha}(\Omega_\pm), h^{1+\alpha}(\mathbb{S}))).
\end{equation}
Finally, let $\mathcal{B}(f):\mbox{\it{buc}}\,^{2+\alpha}(\Omega_+)\to h^{1+\alpha}(\mathbb{S})$ be given by
\begin{align*}
\mathcal{B}(f)v_+:=\frac{1}{1-f}{\mathop{\rm tr}}_1\partial_yv_+, \quad v_+\in \mbox{\it{buc}}\,^{2+\alpha}(\Omega_+),
\end{align*}
with ${\mathop{\rm tr}}_1$ the trace operator with respect to $\Gamma_1, $ that is  ${\mathop{\rm tr}}_1 v_+(x):=v_+(x,1)$ for all $x\in\mathbb{S}$
and $v_+\in\mbox{\it{BUC}}\,(\Omega_+).$
Clearly, we also have
\begin{equation}\label{anaB1}
\mathcal{B}\in C^\omega(\mathcal{V}, \mathcal{L}(\mbox{\it{buc}}\,^{2+\alpha}(\Omega_+), h^{1+\alpha}(\mathbb{S}))).
\end{equation}

Letting $v_\pm=\phi_{f\pm}^* u_\pm$, we see that problem \eqref{eq:S} is equivalent to  the following system
\begin{equation}\label{eq11}\left\{
\begin{array}{rllllll}
\mathcal{A}_\pm(f) v_\pm&=&0&\text{in}& \Omega_\pm,&0\leq t\leq T,\\[1ex]
\mathcal{B}(f)v_+&=&g_2&\text{on}& \Gamma_1,&0\leq t\leq T,\\[1ex]
 v_-&=&g_1&\text{on}&\Gamma_{-1},&0\leq t\leq T,\\[1ex]
v_+-v_-&=&\gamma \kappa(f)+\varpi f&\text{on}& \Gamma_0,&0\leq t\leq T,\\[1ex]
\partial_t f+\mathcal{B}_\pm(f)v_\pm&=&0 &\text{on}& \Gamma_{0},&0< t\leq T,\\[1ex]
f(0)&=&f_0,&&
\end{array}
\right.
\end{equation}
where $\kappa: h^{4+\alpha}(\mathbb{S})\to h^{2+\alpha}(\mathbb{S}),$  with
\[
\kappa(f)=\frac{f''}{(1+f'^2)^{3/2}}\quad\text{for}\quad f\in h^{4+\alpha}(\mathbb{S}),
\]
is the pulled-back curvature operator.
Obviously, $\kappa$ depends analytically on $f$ as well, i.e. $\kappa\in C^\omega(h^{4+\alpha}(\mathbb{S}),h^{2+\alpha}(\mathbb{S})).$
The notion of classical solution for \eqref{eq11} is defined similarly to that for \eqref{eq:S}.
The systems \eqref{eq:S} and \eqref{eq11} are equivalent in the following sense:

\begin{lemma}\label{L:equivalent} Let $(f,u_+,u_-)$ be   solution of \eqref{eq:S}. 
The tuple $(f,\phi_{f+}^*u_+,\phi_{f-}^*u_-)$ is then a  solution to \eqref{eq11}.
Vice versa, if $(f,v_+,v_-)$ is a classical solution  of  \eqref{eq11}, then $(f,\phi_*^{f+}v_+,\phi_*^{f-}v_-)$
 is a classical solution of \eqref{eq:S}.
\end{lemma}
\begin{proof}
The main difficulty lies in showing that, if $f\in\mathcal{V}$, then   $u_\pm\in \mbox{\it buc}\,^{2+\alpha}(\Omega_\pm(f))$ exactly when
$v_\pm\in \mbox{\it buc}\,^{2+\alpha}(\Omega_\pm).$
However, this equivalence  can be proved by using the mean value theorem as in \cite[Lemma 1.2]{EM1}.
\qed
\end{proof}

We emphasize  that the mapping $f,$ describing the moving boundary separating the fluids, is preserved by this transformation.
Although the transformation we made above has the draw-back of introducing additional nonlinear coefficients, it allows us  to 
rewrite the problem as a nonlinear operator equation.

\subsection{The operator equation}
We now define  solution operators to mixed boundary value problems which are closely related to our system \eqref{eq11}.
Composing these operators, we obtain thereafter an  operator equation which is equivalent to the formulation \eqref{eq11}.
Moreover, the operators defined below will reveal that it is  crucial to determine the mapping $f$, describing  the evolution of the interface.
The potentials $v_\pm$ are then the image of $f$ under these operators.   
For this reason, we refer also only to $f$ as being a solution of \eqref{eq11}.

Given $f\in \mathcal{V}$ and $(q,h)\in h^{1+\alpha}(\mathbb{S})\times h^{2+\alpha}(\mathbb{S})$,  
let $\mathcal{T}(f,q,h)\in \mbox{\it{buc}}\,^{2+\alpha}(\Omega_-)$ denote the solution of the linear, elliptic mixed boundary value problem
\begin{equation}\label{eq12}\left\{
\begin{array}{rlllll}
\mathcal{A}_-(f) v_-&=&0&\text{in}& \Omega_-,\\[1ex]
\mathcal{B}_-(f)v_-&=&q &\text{on}& \Gamma_{0},\\[1ex]
v_-&=&h&\text{on}& \Gamma_{-1}.
\end{array}
\right.
\end{equation}
Further on, we define  
$\mathcal{S}:\mathcal{V}\times h^{1+\alpha}(\mathbb{S})\times h^{2+\alpha}(\mathbb{S})\to\mbox{\it{buc}}\,^{2+\alpha}(\Omega_+)$
 by writing 
$\mathcal{S}(f,q,h)$ for the unique solution of the problem
\begin{equation}\label{eq13}\left\{
\begin{array}{rlllll}
\mathcal{A}_+(f) v_+&=&0&\text{in}& \Omega_+,\\[1ex]
\mathcal{B}(f)v_+&=&q&\text{on}& \Gamma_1,\\[1ex]
 v_+&=&h+\gamma \kappa(f)+\varpi f&\text{on}&\Gamma_{0}.
\end{array}
\right.
\end{equation}

With this notation we observe that $(f,v_+,v_-)$ is a solution of \eqref{eq11} if and only if $f(0)=f_0,$ $v_-=\mathcal{T}(f,-\partial_tf,g_1),$
$v_+=\mathcal{S}(f,g_2,{\mathop{\rm tr}\,}  v_-) $, and 
\begin{equation}\label{eq:CPP}
\partial_tf+\mathcal{B}_+(f)\mathcal{S}(f,g_2,{\mathop{\rm tr}\,} \mathcal{T}(f,-\partial_tf,g_1))=0.
\end{equation}
We have obtained in this way an operator equation only for the mapping $f$ which has to be solved.
A possible approach to \eqref{eq:CPP} is to use the implicit function theorem and Newton's iteration method, similarly as in \cite{Y1}.
We proceed differently by decomposing  the operators $\mathcal{T}$ and $\mathcal{S}$ appropriately so that  we may rewrite the equation \eqref{eq:CPP}
as an abstract Cauchy problem  which has order $1+2{\mathop{\rm sign}\,}(\gamma).$

Let us start by noticing that $\mathcal{T}(f,q, h)=\mathcal{T}_1(f)h+\mathcal{T}_2(f)q$,
 where, given $f\in\mathcal{V},$ the linear operators $\mathcal{T}_1(f)$ and $\mathcal{T}_2(f)$ are defined by
\[
\mathcal{T}_1(f)h:=(\mathcal{A}_-(f),\mathcal{B}_-(f),{\mathop{\rm tr}})^{-1}(0,0,h), \, \, \mathcal{T}_2(f)q:=(\mathcal{A}_-(f),\mathcal{B}_-(f),{\mathop{\rm tr}})^{-1}(0,q,0). 
\] 
 Similarly, we decompose $\mathcal{S}(f,q,h)=\mathcal{S}_1(f)h+\mathcal{S}_2(f)q+\mathcal{S}_3(f),$  with 
\begin{align*}
&\mathcal{S}_1(f)h:=(\mathcal{A}_+(f),\mathcal{B}(f),{\mathop{\rm tr}})^{-1}(0,0,h),\quad \mathcal{S}_2(f)q:=(\mathcal{A}_+(f),\mathcal{B}(f),{\mathop{\rm tr}})^{-1}(0,q,0),\\[1ex]
&\mathcal{S}_3(f):=(\mathcal{A}_+(f),\mathcal{B}(f),{\mathop{\rm tr}})^{-1}(0,0,\gamma\kappa(f)+\varpi f).
\end{align*} 
In virtue of \eqref{eq10}-\eqref{anaB1},  and taking also into consideration that the operator mapping a bijective linear operator onto its inverse is analytical,  
 we obtain that the linear operators $\mathcal{T}_i(f),$ $\mathcal{S}_i(f),$ $i=1,2,$ as well as $\mathcal{S}_3(f),$ depend analytically on $f$. 
 With this notation, problem \eqref{eq:CPP} rewrites
 \begin{align*}
 0=&\partial_tf+\mathcal{B}_+(f)\mathcal{S}(f,g_2,{\mathop{\rm tr}\,} \mathcal{T}_1(f)g_1-{\mathop{\rm tr}\,} \mathcal{T}_2(f)\partial_tf))\\[2ex]
 =&\left[{\mathop{\rm id}\,}_{h^{1+\alpha}(\mathbb{S})}-\mathcal{B}_+(f)\mathcal{S}_1(f){\mathop{\rm tr}\,} \mathcal{T}_2(f)\right]\partial_tf+\mathcal{B}_+(f)\mathcal{S}_3(f)+\mathcal{B}_+(f)\mathcal{S}_1(f){\mathop{\rm tr}\,}\mathcal{T}_1(f)g_1\\[1ex]
&+\mathcal{B}_+(f)\mathcal{S}_2(f)g_2.
 \end{align*}

 The new expression of \eqref{eq:CPP} is more involved, but we remark that the time derivative $\partial_tf$ appears in the equality above as the argument of an operator belonging to $\mathcal{L}(h^{1+\alpha}(\mathbb{S})).$
 This operator is  invertible for all $f\in\mathcal{V}.$

 \begin{lemma} \label{L:1}
Given $f\in\mathcal{V},$ the operator
$\mathcal{G}(f):={\mathop{\rm id}\,}_{h^{1+\alpha}(\mathbb{S})}-\mathcal{B}_+(f)\mathcal{S}_1(f){\mathop{\rm tr}\,} \mathcal{T}_2(f)$ 
is an isomorphism, i.e.   $\mathcal{G}(f)\in{\it Isom}(h^{1+\alpha}(\mathbb{S})).$
\end{lemma}
\begin{proof} Let $p, q\in h^{1+\alpha}(\mathbb{S})$ be given such that $\mathcal{G}(f)q=p$.
If $u_-:=(\mathcal{T}_2(f)q)\circ\phi_{f-}^{-1}$ and $u_+:=(\mathcal{S}_1(f){\mathop{\rm tr}\,} \mathcal{T}_2(f)q)\circ\phi_{f+}^{-1},$ then $(u_-,u_+)$ solves
the following system
\begin{equation}\label{eq:S1S}
\left\{
\begin{array}{rllllll}
\Delta u_\pm&=&0&\text{in}& \Omega_\pm(f), \\[1ex]
u_-&=&0&\text{on}&\Gamma_{-1},\\[1ex]
k\mu_-^{-1}\partial_\nu u_-&=&q/\sqrt{1+f'^2}&\text{on}&\Gamma(f),\\[1ex]
u_--u_+&=&0  &\text{on}& \Gamma(f),\\[1ex]
\partial_y u_+&=&0&\text{on}& \Gamma_1.
\end{array}
\right. 
\end{equation} 
With this notation   $\mathcal{G}(f)q=p$ is  equivalent to 
\begin{equation}\label{eq:S2S}
k\mu_+^{-1}\partial_\nu u_+=(q-p)/\sqrt{1+f'^2}\quad\text{on}\quad\Gamma(f).
\end{equation}
Given $p\in C^{1+\alpha}(\mathbb{S}),$ consider the diffraction problem which is obtained by eliminating $q$ from  \eqref{eq:S2S} and the fourth equation of \eqref{eq:S1S}
\begin{equation}\label{eq:S3S}
\left\{
\begin{array}{rllllll}
\Delta u_\pm&=&0&\text{in}& \Omega_\pm(f), \\[1ex]
u_-&=&0&\text{on}&\Gamma_{-1},\\[1ex]
k\mu_-^{-1}\partial_\nu u_- -k\mu_+^{-1}\partial_\nu u_+&=&p/\sqrt{1+f'^2}&\text{on}&\Gamma(f),\\[1ex]
u_--u_+&=&0  &\text{on}& \Gamma(f),\\[1ex]
\partial_y u_+&=&0&\text{on}& \Gamma_1.
\end{array}
\right. 
\end{equation} 
We infer from maximum principles for elliptic problems that,  when $p=0$,  \eqref{eq:S3S} has only the trivial solution $(u_-,u_+)=(0,0).$
Whence, for general $p\in C^{1+\alpha}(\mathbb{S}),$ \eqref{eq:S3S} has a unique solution $(u_-,u_+)\in \mbox{\it{BUC}}\,^{2+\alpha}(\Omega_-)\times\mbox{\it{BUC}}\,^{2+\alpha}(\Omega_+)$  
(see \cite[ Theorems 1.61 and  16.2]{La}).
Consequently, we find from \eqref{eq:S2S} a unique $q\in C^{1+\alpha}(\mathbb{S})$ such that $\mathcal{G}(f)q=p.$
Taking into consideration that $\mathcal{G}$ depends analytically on $f$, the desired result follows from a density argument. 
\qed
\end{proof}

Applying the inverse $\mathcal{G}(f)^{-1}$ to the equation \eqref{eq:CPP},  we obtain an equivalent formulation of our original system \eqref{eq:S}:
 \begin{prop}\label{P:parabolic} 
A function $f$ is a classical H\"older solution of  problem \eqref{eq:S} if and only if
 $f$ is a solution of the evolution equation
\begin{equation}\label{eq:CP}
\partial_tf=\Phi(t,f),\qquad f(0)=f_0,
\end{equation}
where $\Phi:[0,\infty)\times\mathcal{V}\to h^{1+\alpha}(\mathbb{S})$, $\Phi(t,f):=\Phi_1(f)+\Phi_2(t,f),$ is  defined by
\begin{align}
\label{eq:Phi1}
&\Phi_1(f):=-\left(\mathcal{G}(f)\right)^{-1}\mathcal{B}_+(f)\mathcal{S}_3(f),\\[2ex]
\label{eq:Phi2}
&\Phi_2(t,f):=-\left(\mathcal{G}(f)\right)^{-1}\left[\mathcal{B}_+(f)\mathcal{S}_1(f){\mathop{\rm tr}\,} \mathcal{T}_1(f)g_1(t)+\mathcal{B}_+(f)\mathcal{S}_2(f)g_2(t)\right].
\end{align}
for all $ (t,f)\in[0,\infty)\times\mathcal{V}.$
\end{prop}
 
 \section{The proof of Theorem~\ref{T:main}}
 We use in this section the abstract theory for evolution equations as presented in \cite{L} to prove the results stated in Theorem \ref{T:main}.
 Let us first emphasize that the operator $\Phi$ is analytic with respect to $f$ and the partial derivatives with respect to this variable are continuous in both variables $(t,f).$
 In order to prove the well-posedness of the problem \eqref{eq:S}, we are left to check that  the Fr\'echet derivative $\partial_f\Phi(0,0)$  generates a strongly continuous and analytic semigroup.
 The following lemma is, besides Proposition \ref{P:parabolic}, the essential part in the proof of Theorem \ref{T:main}.
 \begin{lemma}\label{L:esse}
 If $g_i(0)=:c_i\in\mathbb{R}, i=1,2,$ then the  derivative $\partial_f\Phi(0,0)$ is a Fourier multiplier of order $1+2{\mathop{\rm sign}\,}(\gamma)$.
More exactly, given $f\in h^{2+2{\mathop{\rm sign}\,}(\gamma)+\alpha}(\mathbb{S})$ let $f=\sum_{m\in\mathbb{Z}}a_me^{imx}$ denote its Fourier series expansion.
Then, we have 
 \begin{equation}\label{eq:final}
\text{$\partial_f\Phi(0,0)f=\sum_{m\in\mathbb{Z}}\lambda_ma_me^{imx},$ }
\end{equation}
where
\begin{equation}\label{eq:final2}
\lambda_m:=\frac{k|m|\tanh(|m|)}{\mu_++\mu_-\tanh^2(m)}\left[c_2\left(\frac{\mu_-}{\mu_+}-1\right)+\varpi-\gamma m^2\right],\quad m\in\mathbb{Z}.
\end{equation}
 \end{lemma}

In the main body of this section we prove Lemma \ref{L:esse} and, at the end of the section, we combine our results
to prove existence, uniqueness, and dependence of  solutions to \eqref{eq:S} on the initial data and time.
To this scope, we first determine  an expansion for the isomorphism $\mathcal{G}(0)$. 
Given   $q\in h^{1+\alpha}(\mathbb{S}) $, we
consider its  Fourier series
$q=\sum_{m\in\mathbb{Z}}a_me^{imx}.$ 
As in \cite{EM1}, we obtain  that $\mathcal{T}_2(0)q$  has the following   expansion 
\[
\mathcal{T}_2(0)q=\frac{\mu_-}{k}(1+y)a_0
+\frac{\mu_-}{k}\sum_{m\in\mathbb{Z}\setminus\{0\}}\left(\frac{e^{m}e^{my}}{m(e^{m}+e^{-m})}-\frac{e^{-m}e^{-my}}{m(e^{m}+e^{-m})}\right)a_m e^{imx}.
\]
Furthermore, given $h\in h^{2+\alpha}(\mathbb{S})$ with $ h=\sum_{m\in\mathbb{Z}}b_me^{imx}$, it holds that 
\begin{equation}\label{eq:SYS}
\mathcal{S}_1(0)h=\sum_{m\in\mathbb{Z}}\left(\frac{e^{-m}e^{my}}{e^{m}+e^{-m}}
+\frac{e^{m}e^{-my}}{e^{m}+e^{-m}}\right)b_me^{imx}, 
\end{equation}
and so
\[
\mathcal{S}_1(0){\mathop{\rm tr}\,} \mathcal{T}_2(0)q=\frac{\mu_-}{k}\sum_{m\in\mathbb{Z}}\left(\frac{e^{-m}e^{my}}{e^{m}+e^{-m}}
+\frac{e^{m}e^{-my}}{e^{m}+e^{-m}}\right)\frac{\tanh(m)}{m}a_me^{imx}.
\]
We set by convention $\tanh(m)/m=1$ if $m=0.$
Summarising, we obtain that
\begin{equation}\label{eq:G_1}
\mathcal{G}(0)q=\sum_{m\in\mathbb{Z}}\left(1+\frac{\mu_-}{\mu_+}\tanh^2(m)\right)a_me^{imx}
\end{equation}
for all  $q\in h^{1+\alpha}(\mathbb{S}) $ with  $q=\sum_{m\in\mathbb{Z}}a_me^{imx}.$

 We begin by studying the mapping $\Phi_1$. 
 Taking into consideration that
 $\mathcal{S}_3(0)=0$, we obtain that 
 \begin{equation}\label{eq:P1}
 \partial\Phi_1(0)=-(\mathcal{G}(0))^{-1}\mathcal{B}_+(0)\partial\mathcal{S}_3(0),
 \end{equation}
 where for $f\in h^{2+2{\mathop{\rm sign}\,}(\gamma)+\alpha}(\mathbb{S}),$ we have  
$\partial\mathcal{S}_3(0)[f]=(\Delta,\mathcal{B}(0),{\mathop{\rm tr}\,})^{-1}(0,0,\gamma f''+\varpi f ).$ 
 The last relation follows in virtue of  $\mathcal{A}_+(0)=\Delta$ and $\partial\kappa(0)[f]=f''$.
Reconsidering \eqref{eq:SYS}, we see due to \eqref{eq:G_1} that  
  $\partial\Phi_1(0)$ is a Fourier multiplication operator with
 \begin{equation}\label{eq:mi}
  \partial\Phi_1(0)[f]=\sum_{m\in\mathbb{Z}}\frac{k\tanh(|m|)}{\mu_++\mu_-\tanh^2(m)}|m|(\varpi-\gamma m^2)a_me^{imx}
 \end{equation}
 for  $f=\sum_{m\in\mathbb{Z}}a_me^{imx}\in h^{2+2{\mathop{\rm sign}\,}(\gamma)+\alpha}(\mathbb{S}).$
 This representation  is sufficient to obtain the well-posedness of the problem when considering surface tension effects.
 However, when  surface tension is neglected, or when we study the stability properties of  equilibria, it is necessary
to analyse more closely the partial derivative $\partial_f\Phi_2(0,0).$

In the following we set   $g_i(0)=c_i$, $i=1,2,$ where $c_i\in\mathbb{R}$. 
By the chain rule we get 
 \begin{align*}
 \partial_f\Phi_2(0,0)[f]=&-\left(\mathcal{G}(0)\right)^{-1}\mathcal{B}_+(0)\partial\mathcal{S}_2(0)[f]c_2-\left(\mathcal{G}(0)\right)^{-1}\partial\mathcal{B}_+(0)[f]\mathcal{S}_2(0)c_2\\[1ex]
 &-\partial\left(\left(\mathcal{G}(\cdot)\right)^{-1}\right)(0)[f]\mathcal{B}_+(0)\mathcal{S}_2(0)c_2=:I_1f+I_2f+I_3f,
 \end{align*}
 for  $f\in h^{2+2{\mathop{\rm sign}\,}(\gamma)+\alpha}(\mathbb{S}).$
This is due to the fact that  $\mathcal{T}_1(f)c_1=c_1,$ $\mathcal{S}_1(f)c_1=c_1,$ relations which imply that  the first term of $\Phi_2(0,f)$ is the zero function:
 \[
-\left(\mathcal{G}(f)\right)^{-1}\mathcal{B}_+(f)\mathcal{S}_1(f)\mathcal{T}_1(f)c_1=0, \qquad\forall f\in\mathcal{V}.
\]
 
We  determine next expansions for the linear operators $I_i$, $i\in\{1,2,3\}.$
To this scope we note that
 \begin{align*}
 \partial\mathcal{A}_\pm(0)[f]&=-2(1\mp y)f'\frac{\partial^2}{\partial x\partial y}\pm 2f\frac{\partial^2}{\partial y^2}-(1\mp y)f''\frac{\partial}{\partial y},\\[1ex]
 \partial\mathcal{B}_\pm(0)[f]&=\frac{k}{\mu_\pm}\left(\pm f{\mathop{\rm tr}\,} \frac{\partial}{\partial y}-f'{\mathop{\rm tr}\,}  \frac{\partial}{\partial x}\right), \quad \text{and} \quad 
\partial\mathcal{B}(0)[f]=f{\mathop{\rm tr}}_1\frac{\partial}{\partial y}
 \end{align*}
 for all $ f\in h^{2+2{\mathop{\rm sign}\,}(\gamma)+\alpha}(\mathbb{S}).$

 Let  $ f\in h^{2+2{\mathop{\rm sign}\,}(\gamma)+\alpha}(\mathbb{S})$ be given and denote by $f=\sum_{m\in\mathbb{Z}} a_m e^{imx}$ its Fourier series.
We begin determining  the expansion corresponding to $I_1f.$
The function  $\partial\mathcal{S}_2(0)[f]c_2$ is the solution of the boundary value problem
\begin{equation*}
\left\{
\begin{array}{rlllll}
\Delta w&=&-\partial\mathcal{A}_+(0)[f]\mathcal{S}_2(0)c_2&\text{in}& \Omega_+,\\[1ex] 
\partial_yw&=&-\partial\mathcal{B}(0)[f]\mathcal{S}_2(0)c_2&\text{on}& \Gamma_1,\\[1ex]
 w&=&0&\text{on}&\Gamma_{0},
\end{array}
\right.
\end{equation*} 
where $\mathcal{S}_2(0)c_2=c_2y$ in $\Omega_+.$
We make the following Fourier series ansatz
 \[
 \partial\mathcal{S}_2(0)[f]c_2=\sum_{m\in\mathbb{Z}} w_m(y)e^{imx}.
\]
 Plugging  this expression into the system found above, we are left, after identifying the coefficients of $e^{imx}$, to solve the problems
 \begin{equation}\label{eq:how0}
\left\{
\begin{array}{rlllll}
 w_m''-m^2w_m&=&c_2a_m m^2 (y-1)&\text{in}& 0<y<1,\\[1ex]
w_m'(1)&=&-c_2a_m&\text{on}& \Gamma_1,\\[1ex]
 w_m(0)&=&0&\text{on}&\Gamma_{0}.
\end{array}
\right.
\end{equation} 
The solution of \eqref{eq:how0} is given by 
 \begin{align*}
w_m(y)=&c_2a_m\tanh(m)\frac{e^{my}-e^{-my}}{2}-c_2a_m\frac{e^{my}+e^{-my}}{2}+c_2a_m(1-y)
\end{align*}
for all $m\in\mathbb{N}$ and $y\in[0,1].$
 Whence,
 \[
\mathcal{B}_+(0)w=\frac{k}{\mu_+}\partial_\nu w=\sum_{m\in\mathbb{Z}}\frac{k}{\mu_+}(m\tanh(m)-1)c_2a_me^{imx}
\]
 and we obtain, in view of \eqref{eq:G_1},  the following expansion for $I_1f$
 \[
I_1f=-\sum_{m\in\mathbb{Z}}\frac{k(m\tanh(m)-1)}{\mu_++\mu_-\tanh^2(m)}c_2a_me^{imx}.
\]
 Since $\mathcal{S}_2(0)c_2=c_2y$ in $\Omega_+,$ a simple computation yields
  \[
I_2f=-\sum_{m\in\mathbb{Z}}\frac{k}{\mu_++\mu_-\tanh^2(m)}c_2a_me^{imx}.
\]

Lastly, we look for the Fourier  series of the function
$I_3f.$ 
To simplify notation we set $\mathcal{F}:=\mathcal{G}^{-1}$ and put $c:=kc_2/\mu_+.$
In view of $\mathcal{F}(f)\mathcal{G}(f)={\mathop{\rm id}\,}_{h^{1+\alpha}(\mathbb{S})},$  the chain rule yields
$\partial\mathcal{F}(0)[f]c=-\left(\mathcal{G}(0)\right)^{-1}\partial\mathcal{G}(0)[f]\left(\mathcal{G}(0)\right)^{-1}c,$
while \eqref{eq:G_1} implies that $\mathcal{G}(0)c=c$.
From the definition of $\mathcal{G},$ we get by differentiation that
\begin{align*}
\partial\mathcal{G}(0)[f]c=&-\partial\mathcal{B}_+(0)[f]\mathcal{S}_1(0){\mathop{\rm tr}\,} \mathcal{T}_2(0)c-\mathcal{B}_+(0)\partial\mathcal{S}_1(0)[f]{\mathop{\rm tr}\,} \mathcal{T}_2(0)c\\[1ex]
&-\mathcal{B}_+(0)\mathcal{S}_1(0){\mathop{\rm tr}\,} \partial\mathcal{T}_2(0)[f]c=:J_1f+J_2f+J_3f.
\end{align*}
 
We start by analysing $J_3.$ 
Differentiating \eqref{eq12} with respect to $f$, we conclude that  $\partial\mathcal{T}_2(0)[f]c$ is the solution of the linear elliptic problem
\begin{equation*}
\left\{
\begin{array}{rlllll}
\Delta w&=&-\partial\mathcal{A}_-(0)[f]\mathcal{T}_2(0)c=c\mu_-k^{-1}f''(1+y)&\text{in}& \Omega_-,\\[1ex]
\partial_yw&=&-\mu_-k^{-1}\partial\mathcal{B}_-(0)[f]\mathcal{T}_2(0)c=f\partial_y(\mathcal{T}_2(0)c)=c\mu_-k^{-1}f&\text{on}& \Gamma_0,\\[1ex]
 w&=&0&\text{on}&\Gamma_{-1},
\end{array}
\right.
\end{equation*} 
where we made use of the relation $\mathcal{T}_2(0)c=c\mu_-k^{-1}(1+y).$ 
Expanding
 \[
w(x,y)=\sum_{m\in\mathbb{Z}} w_m(y)e^{imx},\qquad (x,y)\in\Omega_-,
\]
we get that $w_m$ is the solution of the system
 \begin{equation}\label{eq:how}
\left\{
\begin{array}{rlllll}
 w_m''-m^2w_m&=&-c\mu_-k^{-1}a_m m^2 (1+y)&\text{in}& -1<y<0,\\[1ex]
w_m'(0)&=&c\mu_-k^{-1}a_m,\\[1ex]
 w_m(-1)&=&0.
\end{array}
\right.
\end{equation} 
Since $w_m(y)=c\mu_-k^{-1}a_m(1+y),$ we find that
\[
\partial\mathcal{T}_2(0)[f]c|_{y=0}=\sum_{m\in\mathbb{Z}}c\mu_-k^{-1}a_me^{imx}.
\]
 Together with  \eqref{eq:SYS}, we then get
\[
\mathcal{S}_1(0){\mathop{\rm tr}\,} \partial\mathcal{T}_2(0)[f]c=\sum_{m\in\mathbb{Z}}c\mu_-k^{-1}\left(\frac{e^{-m}e^{my}}{e^{m}+e^{-m}}
+\frac{e^{m}e^{-my}}{e^{m}+e^{-m}}\right)a_me^{imx}.
\]
Whence, we have shown that $J_3$ is a Fourier multiplier of the following form
\[
J_3\sum_{m\in\mathbb{Z}}a_me^{imx}=\sum_{m\in\mathbb{Z}\setminus\{0\}}\frac{c\mu_-}{\mu_+}m\tanh(m)a_me^{imx}.
\]
To find an expansion for $J_2f$ we note first  that $\mathcal{T}_2(0)c=\mu_-c(1+y),$ which leads to $\mathcal{S}_1(f){\mathop{\rm tr}\,} \mathcal{T}_2(0)c=c\mu_-$ for all $f\in\mathcal{V}$.
This  implies that 
 $J_2f$ is the zero function.
Since $\mathcal{S}_1(0){\mathop{\rm tr}\,} \mathcal{T}_2(0)c=\mu_-c$, we finally get $J_1f=0.$
We have thus shown that $\partial\mathcal{G}(0)[f]c=J_3f$ 
for all $f\in h^{2+2{\mathop{\rm sign}\,}(\gamma)+\alpha}(\mathbb{S})$.
Summarising, we find the following relation 
\begin{equation}\label{eq:grea}
\partial_f\Phi_2(0,0)[f]=\sum_{m\in\mathbb{Z}}\left(\frac{\mu_-}{\mu_+}-1\right)\frac{km\tanh(m)}{\mu_++\mu_-\tanh^2(m)}c_2a_me^{imx},
\end{equation}
which leads, together with \eqref{eq:mi}, to \eqref{eq:final}.

We are now prepared to prove  the main result stated in Theorem \ref{T:main}.
The argumentation strongly   relies  on the regularity of $\Phi,$ relation \eqref{eq:final}, and well-known interpolation properties of the 
small H\"older spaces
\begin{equation}\label{interpolation}
(h^{\sigma_0}(\mathbb{S}),h^{\sigma_1}(\mathbb{S}))_\theta=h^{(1-\theta)\sigma_0+\theta\sigma_1}
(\mathbb{S}),
\end{equation}
if $  \theta\in(0,1)$ \text{and} $(1-\theta)\sigma_0+\theta\sigma_1\notin\mathbb{N}.$

\begin{proof}[Proof of Theorem~{\rm \ref{T:main}}] 
Assume first  $\gamma=0.$ 
Using  \cite[Theorem 3.4]{EM1} we obtain that the Fr\'echet derivative
$\partial_f\Phi(0,0)$, which has order 1, generates a strongly continuous and analytic semigroup in 
$\mathcal{L}(h^{1+\beta}(\mathbb{S}))$ for all $\beta\in(0,1),$
if $g_i(0)=c_i\in\mathbb{R}$, $i=1,2,$ and \eqref{eq:cond} is satisfied.
More precisely, with  the notation  used in \cite{Am}, this is written as $-\partial_f\Phi(0,0)\in\mathcal{H}(h^{2+\beta}(\mathbb{S}), h^{1+\beta}(\mathbb{S})).$

Pick some $\beta<\alpha.$
Since $\mathcal{H}(h^{2+\beta}(\mathbb{S}), h^{1+\beta}(\mathbb{S}))$ is an open subset of the space of bounded and linear operators
$\mathcal{L}(h^{2+\beta}(\mathbb{S}), h^{1+\beta}(\mathbb{S}))$ we find by continuity  open neighbourhoods of the zero function
$\mathcal{O}_i\subset h^{i+\alpha}(\mathbb{S})$, $i\in\{1,2\}$, and $\mathcal{O}\subset\mathcal{V}$ 
such that 
$-\partial_f\Phi(t,f)\in\mathcal{H}(h^{2+\beta}(\mathbb{S}), h^{1+\beta}(\mathbb{S}))$ 
for all $f\in\mathcal{O}$ and $g_i(t)\in c_i+\mathcal{O}_i.$
The existence result  follows now in virtue of \cite[Theorem 8.4.1]{L} and relation \eqref{interpolation}.
The regularity assertion is obtained by using  \cite[Corollary 8.4.6]{L}.

When $\gamma>0$  there are no restrictions on the smallness of $g_i, 1\leq i\leq2.$
This is due to the fact that the derivative $\partial_f\Phi(t,f) $ is a third order operator, and all  terms not containing $\gamma$ are treated as lower order perturbations.
In view of \cite[Proposition 2.4.1]{L} we obtain the desired local well-posedness result.

If  $\gamma=0,$ $g_i\in\mathbb{R}, i=1,2,$ and \eqref{eq:cond2} is fulfilled, then the linearised problem 
\[
\partial_th=\partial_f\Phi(0,0)h,\qquad h(0)=0, 
\]
is ill-posed in the sense of Hadamard in $h^{1+\alpha}(\mathbb{S}),$ since $\partial_f\Phi(0,0)$ does not generate a strongly continuous  semigroup in this case.
This completes the proof.
\qed
\end{proof}

\section{Stability properties}\label{S:6}
In this section we assume that $g_1\equiv const.$ and $g_2\equiv 0.$
Then, the flat interface $f_*\equiv0$ is, in view of \eqref{eq:sim}, a steady-state solution with the constant pressure distributions $u_-=u_+=g_1.$  
Note that  in this case we deal with an autonomous equation since the operator $\Phi$ in \eqref{eq:CP} does no longer depend on the time variable $t$ ($\Phi_2=0$). 
In the following we are interested in the stability properties of the equilibrium $f_*.$
Before doing this, we observe that:
\begin{rem}[Conservation of volume]\label{R:1} Let $g_1\in\mathbb{R}$ and $g_2\equiv 0.$
If $f$ is a solution of problem \eqref{eq:S} then
\[
\int_{\mathbb{S}} f(t,x)\,dx=const.
\]
as long as the solution to \eqref{eq:S} exists.
\end{rem}
\begin{proof}
See the discussion  below relation \eqref{eq:sim}.
\qed
\end{proof}

Following Remark \ref{R:1}, when studying the stability properties of the equilibrium $f_*$,
we are led to the assumption $f_0\in h^{2+2{\mathop{\rm sign}\,}(\gamma)+\alpha}_0(\mathbb{S}),$
where given $m\in\mathbb{N},$ the space  $h^{m+\alpha}_0(\mathbb{S})$ consists of the functions in $h^{m+\alpha}(\mathbb{S})$   having integral mean equal to $0$.
Hence, we have to restrict our problem to the set $\mathcal{V}_0:=\mathcal{V}\cap h^{4+\alpha}_0(\mathbb{S}).$

\begin{lemma}\label{L:11} Given $f\in\mathcal{V}_0,$ we have that $\Phi(f_0)\in h^{1+\alpha}_0(\mathbb{S}).$
\end{lemma}
\begin{proof}
We begin by showing  that $\mathcal{G}(f)\in\mathcal{L}{is}(h^{1+\alpha}_0(\mathbb{S}))$ for all $f\in\mathcal{V}_0,$ which is equivalent to showing that
$\mathcal{B}_+(f)\mathcal{S}_1(f){\mathop{\rm tr}\,} \mathcal{T}_2(f)q\in h^{1+\alpha}_0(\mathbb{S})$ for all $q\in h^{1+\alpha}_0(\mathbb{S})$ and $f\in\mathcal{V}_0.$
Indeed, setting $v_+:=\mathcal{S}_1(f){\mathop{\rm tr}\,}\mathcal{T}_2(f)q$, we have
\begin{align*}
&\frac{\mu_+}{k}\int_\mathbb{S}\mathcal{B}_+(f)\mathcal{S}_1(f){\mathop{\rm tr}\,}\mathcal{T}_2(f)q\, dx=\int_{\mathbb{S}}(-f'(x),1)\cdot\nabla(\phi_*^{f+}v_+)(x,f(x))\, dx\\[1ex]
&\phantom{space}=\frac{1}{2\pi}\int_{\Gamma(f)}\partial_\nu\left(\phi_*^{f+}v_+\right)\,ds\\[1ex]
&\phantom{space}=\frac{1}{2\pi}\int_{\Omega_+(f)}\Delta\left(\phi_*^{f+}v_+\right)\, dx-\frac{1}{2\pi}\int_{\Gamma_{1}}\partial_\nu\left(\phi_*^{f+}v_+\right)\,ds=0.
\end{align*} 

Lastly, we still have to verify that $\mathcal{B}_+(f)\mathcal{S}_3(f)\in h^{1+\alpha}_0(\mathbb{S})$ for all $f\in\mathcal{V}_0.$
However, this last statement follows by repeating the integration steps presented above.
The desired assertion is now  obtained in virtue of \eqref{eq:Phi1}.
\qed
\end{proof}

We have thus reduced our problem to an autonomous evolution equation 
\begin{equation}\label{eq:auto}
\partial_tf=\Phi(f), \qquad f(0)=f_0,
\end{equation}
where  $\Phi:\mathcal{V}_0\subset h^{2+2{\mathop{\rm sign}\,}(\gamma)+\alpha}_0(\mathbb{S})\to h^{1+\alpha}_0(\mathbb{S}).$
From \eqref{eq:final} and \eqref{eq:final2}, we see that the spectrum of Fr\'echet derivative 
$\partial\Phi(0)\in\mathcal{L}(h^{2+2{\mathop{\rm sign}\,}(\gamma)+\alpha}_0(\mathbb{S}), h^{1+\alpha}_0(\mathbb{S}))$ consists only of the eigenvalues
$\partial\Phi(0)=\{\lambda_m\,:\, 1\leq m\in\mathbb{N}\},$
where  $\lambda_m$ are given by \eqref{eq:final2} with $c_2=0.$
In virtue of \cite[Theorem 9.1.2]{L} and the equivalence of problems \eqref{eq:S} and \eqref{eq:auto} for initial data $f_0\in h^{2+2{\mathop{\rm sign}\,}(\gamma)+\alpha}_0(\mathbb{S})$ we get:

\begin{thm}[Exponential stability] \label{T:main2}
Assume that 
\begin{equation}\label{eq:cond1}
g(\rho_+-\rho_-)<\gamma.
\end{equation} 
Then, the flat equilibrium $f_*\equiv0$ is exponentially stable.
More precisely, given $\omega\in (0,\tanh(1)k(\gamma-\varpi)/(\mu_++\mu_-\tanh(1)))$, there exist positive constants $\delta, M$ such that for all $f_0\in h^{2+2{\mathop{\rm sign}\,}(\gamma)+\alpha}_0(\mathbb{S})$ with $\|f_0\|_{h^{2+2{\mathop{\rm sign}\,}(\gamma)+\alpha}(\mathbb{S})}\leq \delta,$
the solution to \eqref{eq:S} exists in the large and 
\[
\|f(t)\|_{h^{2+2{\mathop{\rm sign}\,}(\gamma)+\alpha}(\mathbb{S})}+\|\partial_tf(t)\|_{h^{1+\alpha}(\mathbb{S})}\leq Me^{-\omega t}\|f_0\|_{h^{2+2{\mathop{\rm sign}\,}(\gamma)+\alpha}(\mathbb{S})}
\]
for all $t\geq 0.$

Moreover, if $g(\rho_+-\rho_-)>\gamma>0,$ this stationary solution is unstable.
\end{thm}

\begin{rem} 
Let us notice that if surface tension effects are considered, then the equilibrium $f_*$ is stable also when the heavier fluid lies above, provided that 
the density jump across  the interface is small compared with the surface tension coefficient, cf. \eqref{eq:cond1}.
If we neglect the surface tension, then the equilibrium is stable only if the heavier fluid  occupies the lower region of the cell.
Moreover, this is also the only possible case which can be analysed, since via \eqref{eq:cond} we must have
$0>g(\rho_+-\rho_-)$ to ensure local well-posedness in this case. 
In the unstable case $g(\rho_+-\rho_-)>\gamma>0$ pattern formation is evidenced in \cite{HLS} by means of numerical simulations.
\end{rem}

\section{Steady-state fingering patterns and instability}

In this last section we have a closer look at the  stationary solutions of 
 problem \eqref{eq:S} under the same constant boundary conditions, $g_1\equiv const.$ and $g_2\equiv 0,$ as in the previous section.
 We still assume that the cell contains equal amounts  of both fluids, meaning that we are interested in determining the stationary solutions of the autonomous problem \eqref{eq:auto}. 
First of all, we notice that if $f\in \mathcal{V}_0$ is a stationary solution of \eqref{eq:S}, then the potentials $u_-$
and $u_+$ are both constant.
This is due to the fact that they are both solutions of elliptic problems with homogeneous Neumann boundary conditions.
We are led by the fifth equation of system \eqref{eq11} 
  to the problem of determining the functions solving
\begin{equation}\label{eq:EQ}
\gamma\frac{f''}{(1+f'^2)^{3/2}}+\varpi f=const.\quad\text{and}\quad\int_{\mathbb{S}} f\, dx=0.
\end{equation}
Recall by \eqref{eq:notation} that $\varpi=g(\rho_+-\rho_-)$ is the constant which measures the density jump across the 
interface 
separating the fluids and   parametrised by the function $f.$
If the density of the  fluid  on the bottom of the cell is greater or equal then that of the fluid above,
 then \eqref{eq:EQ}  has only the trivial  solution $f_*\equiv 0$ (see e.g. \cite{EM5}).
 When $\varpi>0, $ we shall  use a bifurcation argument with the surface tension coefficient $\gamma$ as bifurcation parameter and obtain 
infinitely many global bifurcation branches consisting only of stationary solutions of \eqref{eq:S}.
Of course, in this situation we   consider   the operator $\Phi$ in \eqref{eq:auto} to depend also on $\gamma,$  i.e. $\Phi=\Phi(\gamma,f).$
However it turns out to be   more convenient to treat \eqref{eq:EQ} instead of the operator equation
$\Phi(\gamma,f)=0$ in $ h^{1+\alpha}_0(\mathbb{S}).$

We fix  $\varpi>0$  and look for $(\gamma,f)\in(0,\infty)\times C^2(\mathbb{S})$ with $\|f\|_{C(\mathbb{S})}<1$ which solve \eqref{eq:EQ}. 
Therefore, we introduce an operator which  enables us to consider both equations of \eqref{eq:EQ} at once.
Since all the solutions of \eqref{eq:EQ} are smooth,   for even functions, the problem of finding the solutions of \eqref{eq:EQ}
is equivalent to determining the solutions of the equation
\begin{equation}\label{eq:Up}
\Upsilon(\gamma,f)=0,
\end{equation}
in $(0,\infty)\times\mathcal{U},$ where 
$\mathcal{U}:=\{f\in C^{3+\alpha}_{0,e}(\mathbb{S})\,:\,\|f\|_{C(\mathbb{S})}<1\}$
and the operator $\Upsilon:(0,\infty)\times\mathcal{U}\subset\mathbb{R}\times C^{3+\alpha}_{0,e}(\mathbb{S})\to C^\alpha_{odd}(\mathbb{S})$ is  by  definition  the derivative of 
the left hand side of the first equation of \eqref{eq:EQ}
\[
\Upsilon(\gamma,f):=\gamma \frac{f'''}{(1+f'^2)^{3/2}}-3\gamma\frac{f'f''^2}{(1+f'^2)^{5/2}}+\varpi f'
\]
for $(\gamma,f)\in (0,\infty)\times \mathcal{U}.$
The space $C^{3+\alpha}_{0,e}(\mathbb{S})$ is the subspace of $C^{3+\alpha}(\mathbb{S})$ consisting only of even functions with integral mean $0$, and analogously $C^{\alpha}_{odd}(\mathbb{S})$
consists only of the odd functions in $C^{\alpha}(\mathbb{S}).$
Clearly,  $\Upsilon $ depends analytically  on its variables and $\Upsilon(\gamma,0)=0$ for all $\gamma\in(0,\infty).$  
Its Fr\'echet derivative $\partial_f \Upsilon(\gamma,0)$ is a Fourier multiplication operator with
\begin{equation}\label{eq:FD}
\partial_f\Upsilon(\gamma,0)\left[\sum_{m=1}^\infty a_m\cos(mx)\right]=\sum_{m=1}^\infty \left(\gamma m^3-\varpi m\right)a_m\sin(mx)
\end{equation}
 for all $f=\sum_{m=1}^\infty a_m\cos(mx)\in C^{3+\alpha}_{0,e}(\mathbb{S}).$
Given $l\in\mathbb{N}\setminus\{0\},$ we set 
\begin{equation}\label{eq:ref}
\overline\gamma_l:=\varpi l^{-2}.
\end{equation}
\begin{figure}
$$\includegraphics[width=5cm]{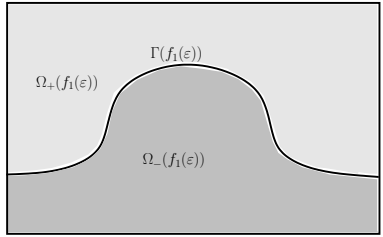}$$
\label{F:B33}
\caption{Steady oil finger penetrating  water.}
\end{figure}
The first result of this section is the following global bifurcation theorem, which states that 
a global bifurcation branch emerges from the trivial flat solution $\{(\gamma,0)\,:\, \gamma>0\}$   at $(\overline\gamma_l,0)$ for all $1\leq l\in\mathbb{N}$,
where $\overline\gamma_l$  is defined by \eqref{eq:ref},
provided $\varpi>0.$
\begin{thm}[Steady-state fingering solutions]\label{T:P3} Let $\varpi>0$ and $l\geq 1$.
 The point $(\overline\gamma_l,0)$ belongs to the closure $\mathcal{R}$ of the set of nontrivial solutions of \eqref{eq:Up} in $(0,\infty)\times\mathcal{U}.$
 Denote by $\mathcal{C}_l$ the connected component of $\mathcal{R}$ to which $(\overline\gamma_l,0)$ belongs.
 Then  $\mathcal{C}_l$ is unbounded in $(0,\infty)\times\mathcal{U}$.
 
 Additionally,  $\mathcal{C}_l$ has, in a small neighbourhood of $(\overline\gamma_l,0),$ an analytic  parame\-trisation
 $(\gamma_l,f_l):(-\delta_l,\delta_l)\to (0,\infty)\times\mathcal{U},$ and 
 \begin{align*} 
\gamma_l(\varepsilon)&=\overline\gamma_l +\frac{3 \varpi}{8}\varepsilon^2+O(\varepsilon^4),\\[1ex]
f_l(\varepsilon)&=\varepsilon\cos(lx)+O(\varepsilon^2),\quad \text{for}\quad \varepsilon\to0. 
 \end{align*}
 
Moreover, any other pair $(\gamma,0)$, with $\gamma>0$, is not a bifurcation point.
\end{thm}
Possible stationary fingering pattern solutions of problem \eqref{eq:S} are pictured in Figure 2.
That $\mathcal{C}_l, $ $l\geq1,$ is unbounded in $(0,\infty)\times\mathcal{U}$ means that either $\mathcal{C}_l$ is unbounded in $\mathbb{R}\times C^{3+\alpha}(\mathbb{S})$, or that $\mathcal{C}_l$ reaches the boundary of $(0,\infty)\times\mathcal{U}.$
It should be mentioned that both situations  may occur \cite{EEM}.
 
\begin{rem}
For all $l\in\mathbb{N},$ $l\geq1,$ we have that $\gamma_l'(0)=0$ and $\gamma_l''(0)>0.$
Consequently, the bifurcation is supercritical.
The bifurcation diagram is pictured in Figure 3.
\begin{figure}
$$\includegraphics[width=7cm]{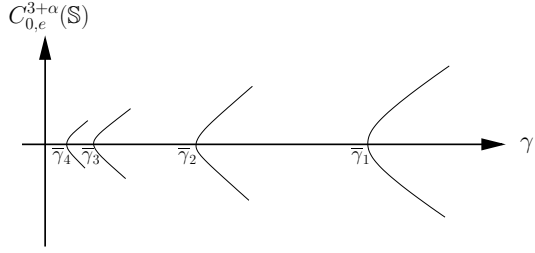}$$
\label{F:h33}
\caption{The  bifurcation diagram.}
\end{figure}
\end{rem}

Concerning the stability properties of these finger-shaped steady-state solution we state:
\begin{thm}[Instability of the fingering patterns]\label{T:instability}
Let $1\leq l\in\mathbb{N}$ be given. 
The stationary solution  $f_l(\varepsilon)$ of problem \eqref{eq:S}, when $\gamma=\gamma_l(\varepsilon)$, is unstable provided $0\neq |\varepsilon|$ is small enough.
\end{thm}

\begin{proof}[Proof of Theorem~{\rm \ref{T:P3}}] 
It is well-known that bifurcation may occur  at $(\gamma,0)$ only if the derivative  $\partial_f\Upsilon(\gamma,0)$ is not an isomorphism, which in view of \eqref{eq:FD},
leads to $\gamma=\overline\gamma_l$ for some $1\leq l\in\mathbb{N}.$
One can easily verify  the assumption of the  theorem on bifurcations from simple eigenvalues due to Crandall and Rabinowitz \cite{CR}, as stated in \cite[Theorem 8.3.1]{BT},
and obtain that analytic bifurcation branches $(\gamma_l,f_l):(-\delta,\delta)\to (0,\infty)\times\mathcal{U}$ consisting entirely of steady-states solutions of \eqref{eq:S}, emerge at $(\overline\gamma_l,0)$  for all $1\leq l\in\mathbb{N}.$ 

In order to study the global behaviour of these branches we rewrite equation \eqref{eq:Up} as follows
\[
f'''-\frac{3f'f''^2}{1+f'^2}+\frac{\varpi}{\gamma}f'(1+f'^2)^{3/2}=0,\qquad (\gamma,f)\in(0,\infty)\times\mathcal{U}.
\]
We use now the property of  the operator $A:C^{3+\alpha}_{0,e}(\mathbb{S})\to C^\alpha_{odd}(\mathbb{S}),$ $Af=f''',$ to be an isomorphism.
By applying the inverse of $A$ to the equation above we  obtain an equivalent formulation for the problem \eqref{eq:Up}
\begin{equation}\label{eq:BGL}
F(\gamma,f)=f+H(\gamma,f)=0, 
\end{equation}
where $H:(0,\infty)\times\mathcal{U}\subset \mathbb{R}\times C^{3+\alpha}_{0,e}(\mathbb{S})\to C^{3+\alpha}_{0,e}(\mathbb{S})$ is the completely continuous operator
\[
H(\gamma,f)=A^{-1}\left(-\frac{3f'f''^2}{1+f'^2}+\frac{\varpi}{\gamma}f'(1+f'^2)^{3/2}\right).
\]
Using Fourier expansions for functions in $C^{3+\alpha}_{0,e}(\mathbb{S})$ as we did before, we conclude  that
the derivative $\partial_f F(\gamma,0),$ has an odd crossing number at $\gamma=\overline\gamma_l$ for all $1\leq l$, cf. \cite[Definition II.3.1]{HK}.
Particularly, the index $i(\partial_fF(\gamma,0),0)$ jumps at $\gamma=\overline\gamma_l$ from $1$ to $-1,$ or vice versa.
We infer from the global Rabinowitz bifurcation theorem \cite[Theorem II.3.3]{HK} that the  
  connected component $\mathcal{C}_l$ of $\mathcal{R}$ to which $(\overline\gamma_l,0)$ belongs, is either unbounded in $(0,\infty)\times\mathcal{U}$ or 
 contains some other bifurcation point $(\overline\gamma_k,0)$, with $k\neq l.$ 
 Using methods from the theory of ordinary differential equations, it is possible to show \cite[Theorems 3.2 and 4.1]{EEM} that the second alternative cannot occur, that is all branches $\mathcal{C}_l$
 are unbounded in $(0,\infty)\times\mathcal{U}$ and pairwise disjoint.

The derivative $\gamma_l'(0)$ is zero  due to the fact that $\Upsilon(\gamma,-f)=0$ for all $(\gamma,f)\in\mathcal{R}.$
Particularly, we have that $f_l(-\varepsilon)=-f_l(\varepsilon)$ and $\gamma_l(\varepsilon)=\gamma_l(-\varepsilon)$ for all $\varepsilon\in(0,\delta).$
We compute now the second derivative $\gamma_l''(0)$.
To this scope, for fixed $\gamma>0,$ we define $\phi:\mathbb{R}^3\to\mathbb{R}$  by the relation
\[
\phi(x,y,z)=\gamma \frac{z}{(1+x^2)^{3/2}}-3\gamma\frac{xy^2}{(1+x^2)^{5/2}}+\varpi x\quad\text{for $(x,y,z)\in\mathbb{R}^3$.}
\]
It holds that $\Upsilon(\gamma,f)=\phi(f',f'',f''')$ for all $f\in\mathcal{U}.$ 
Since $\partial^2_{ij}\phi(0)=0$ for all $1\leq i,j\leq 3$ we get that   $\partial^2_{ff}\Upsilon(\gamma,0)=0$.
Furthermore, the  third order partial derivatives of $\phi$ in $0$  all vanish, except for 
\begin{align*}
&\partial^3_{113}\phi(0)=-3\gamma,\quad \partial^3_{122}\phi(0)=-6\gamma,
\end{align*}
so that $\partial^3_{fff}\Upsilon(\gamma,0)[f,f,f]=-9\gamma f'^2f'''-18\gamma f'f''^2$ 
for all $\gamma>0$ and $f\in C^{3+\alpha}_0(\mathbb{S}).$ 
We may  write for the analytic parametrisation $f_l(\varepsilon)=\varepsilon\cos(lx)+\tau_l(\varepsilon),$
where $\tau_l(0)=\tau_l'(0)=0$, and $\tau_l(\varepsilon) $ belongs to the closed complement of ${\mathop{\rm Ker}\,} \partial_f\Upsilon(\overline\gamma_l,0)$ 
in $C^{3+\alpha}_{0,e}(\mathbb{S}).$
Differentiating  the relation $\Upsilon(\gamma_l(\varepsilon),\varepsilon\cos(lx)+\tau_l(\varepsilon))=0$ three times with respect to $\varepsilon$, at $\varepsilon=0$, yields,
in view of $\gamma_l'(0)=0$ and $\partial^2_{ff}\Upsilon(\overline\gamma_l,0)=0,$ that
\begin{align*} 
&\partial^3_{fff}\Upsilon(\overline\gamma_l,0)[\cos(lx)]^3+3\gamma_l''(0)\partial^2_{\gamma f}\Upsilon(\overline\gamma_l,0)[\cos(lx)]+\partial_f\Upsilon(\overline\gamma_l,0)[\tau'''(0)]=0.
\end{align*}
However, $\partial_f\Upsilon(\overline\gamma_l,0)$ has values  in the complement of $\mathbb{R}\cdot{\sin(lx)}$ in $C^{\alpha}_o(\mathbb{S}),$
so that, by multiplying the relation above by $\sin(lx),$ followed by integration over the unit circle, we get
\begin{align*}
\gamma_l''(0)&=-\frac{1}{3}\frac{\langle\partial^3_{fff}\Upsilon(\overline\gamma_l,0)[\cos(lx)]^3|\sin{lx}\rangle}
{\langle\partial^2_{\gamma f}\Upsilon(\overline\gamma_l,0)[\cos(lx)]|\sin{lx}\rangle}=\frac{3\overline\gamma_l l^2}{4}>0,
\end{align*}
where $\langle f|g\rangle:=\int_{\mathbb{S}} f\overline g\, dx$, $f,g\in L_2(\mathbb{S}),$ is the scalar product on $L_2(\mathbb{S}).$ 
In view of $\overline\gamma_l=\varpi/l^2,$ we conclude that $\gamma_l''(0)=3\varpi/4$ for all $l\geq1.$
This finishes the proof.
\qed
\end{proof}

Finally, we come to the proof of  Theorem \ref{T:instability}.
As in the previous section, the constant boundary conditions on $\Gamma_{\pm1}$ ensure that the operator  $\Phi$, defined by \eqref{eq:CP} (and which depends now also on $\gamma$),
 is independent of time $t$.
The existence of the local bifurcation branches $(\gamma_l,f_l)$ can be also obtained by applying the theorem on bifurcations from simple eigenvalues due to Crandall and Rabinowitz to
the operator equation
\begin{equation}\label{eq:aut*}
\Phi(\gamma,f)=0,\qquad (\gamma,f)\in (0,\infty)\times \mathcal{V}_{0,e}
\end{equation}
 where letting $ h^{m+\alpha}_{0,e}(\mathbb{S})=\{f\in h^{m+\alpha}(\mathbb{S})\,:\, \text{$f$ is even }\}\cap h^{m+\alpha}_0(\mathbb{S}),$
we set $\mathcal{V}_{0,e}:=\mathcal{V}\cap h^{4+\alpha}_{0,e}(\mathbb{S}).$ 
A maximum principle  argument shows  that $\Phi(\gamma,f)$ is even if $f\in\mathcal{V}_{0,e},$ meaning, in virtue of Lemma \ref{L:11}, that
$\Phi:(0,\infty)\times \mathcal{V}_{0,e}\to h^{1+\alpha}_{0,e}(\mathbb{S})$ is a well-defined mapping.
However, it is much more difficult to determine   the derivative   $\gamma_l''(0)$   when applying bifurcation theory to the operator equation \eqref{eq:aut*}.

When proving the instability  of the steady-state  solution $(\gamma_l(\varepsilon),f_l(\varepsilon)),$ $\varepsilon\in(-\delta_l,\delta_l),$ $ 1\leq l,$ for the problem \eqref{eq:auto} it is more accessible  to
show that  they  are unstable steady-state solutions of the restriction 
\begin{equation}\label{eq:CPred}
\partial_tf=\Phi(\gamma,f),\qquad f(0)=f_0,
\end{equation}
where $\Phi:(0,\infty)\times \mathcal{V}_{0,e}\to h^{1+\alpha}_{0,e}(\mathbb{S}).$
Therefore we must study how 
the spectrum of the linearised operator varies along the bifurcation curves.
Our main tool is the exchange of stability theorem of Crandall and Rabinowitz \cite{CR3}.
According to \cite{CR3}, we introduce first some notation.

\begin{defn} Let $\mathbb{E}, \mathbb{F}$ be Banach spaces and $T,K\in\mathcal{L}(\mathbb{E},\mathbb{F}).$
Then $\mu\in\mathbb{R}$ is a $K-$simple eigenvalue of $T$ if 
\[
\dim {\mathop{\rm Ker}\,} (T-\mu K)={\mathop{\rm codim}\,} (T-\mu K)=1,
\]
and, if  ${\mathop{\rm Ker}\,} (T-\mu K)={\mathop{\rm span}\,} \{u_*\},$ then $Ku_*\notin{\mathop{\rm Im}\,} (T-\mu K).$ 
\end{defn} 
The result proved by Crandall and Rabinowitz in \cite{CR3} reads as follows:

\begin{thm} [Crandall-Rabinowitz]\label{T:C-R}
Let $\mathbb{E}, \mathbb{F}$ be Banach spaces, $K\in\mathcal{L}(\mathbb{E},\mathbb{F})$ and $\mathcal{F}:(0,\infty)\times\mathbb{E}\to\mathbb{F}$ of class $C^2$ with $\mathcal{F}(\gamma,0)=0$ near $\gamma_*.$
If $0$ is a $K-$simple eigenvalue of  $T:=\partial_f \mathcal{F}(\gamma_*,0)$ and a $\partial^2_{\gamma f}\mathcal{F}(\gamma_*,0)-$ simple eigenvalue of $T,$ 
then there exists a local curve  $(\gamma(\varepsilon),f(\varepsilon))$ such  that $(\gamma(0), f(0))=(\gamma_*,0)$ and $\mathcal{F}(\gamma(\varepsilon), f(\varepsilon))=0$ 
Moreover, if $\mathcal{F}(\gamma,f)=0$, $f\neq0$, and $(\gamma,f)$ is close to $(\gamma_*,0),$ then $(\gamma,f)=(\gamma(\varepsilon),f(\varepsilon))$ for some $\varepsilon\neq0.$

Furthermore, there are real numbers $\mu(\varepsilon), \lambda(\gamma)$ and vectors $u(\varepsilon), $ $v(\gamma)\in\mathbb{E}$ such that
\begin{align*}
&\partial_f\mathcal{F}(\gamma(\varepsilon),f(\varepsilon))u(\varepsilon)=\mu(\varepsilon)Ku(\varepsilon),\quad \partial_f\mathcal{F}(\gamma,0)v(\gamma)=\lambda(\gamma)Kv(\gamma),
\end{align*}
with $\mu(0)=\lambda(0)=0$ and $u(0)=v(0).$ 
Each curve is $C^1$, with
\[
\lambda'(\gamma_*)\neq0 \quad \text{and} \quad \lim_{\varepsilon\to0, \mu(\varepsilon)\neq0}\frac{-\varepsilon\gamma'(\varepsilon)\lambda'(\gamma_*)}{\mu(\varepsilon)}=1. 
\] 
\end{thm}

\begin{proof}[Proof of Theorem~{\rm \ref{T:instability}}]
In virtue of \eqref{eq:final} and \eqref{eq:final2} we have
\[
\sigma (\partial_f\Phi(\gamma,0))=\left\{\lambda_m(\gamma):=\frac{km\tanh(m)}{\mu_++\mu_-\tanh^2(m)}(\varpi-\gamma m^2)\,:\, 1\leq m\in\mathbb{N}\right\}.
\]
Since $\lambda_1(\overline\gamma_l)>0$ for all $l\geq 2$ we conclude that the spectrum of the linearisation
$\partial_f\Phi(\gamma_l(\varepsilon),f_l(\varepsilon))$ contains positive eigenvalues if $\varepsilon$ is small, and the case $l\geq2$ is proved.

If $l=1,$ then we find ourselves in the critical case when $0$ is the only non-negative point in the spectrum of 
$\partial_f\Phi(\overline\gamma_1,0),$ since $\lambda_1(\overline\gamma_1)=0.$ 
The proof relies now strongly on Theorem \ref{T:C-R}. 
Recall that by \eqref{eq:final}-\eqref{eq:final2} 
the derivative $\partial_f\Phi(\gamma,0)\in\mathcal{L}(h^{4+\alpha}_{0,e}(\mathbb{S}),h^{1+\alpha}_{0,e}(\mathbb{S})),$ is given by
\begin{equation*}
\partial_f\Phi(\gamma,0)\sum_{m=1}^\infty a_m\cos(mx)=\sum_{m=1}^\infty\lambda_m(\gamma)a_m\cos(mx),
\end{equation*}
for all $f=\sum_{m=1}^\infty a_m\cos(mx)\in h^{4+\alpha}_{0,e}(\mathbb{S})$, with $\lambda_m(\gamma)$ as above.
Setting $T=\partial_f\Phi(\overline\gamma_1,0)$ and $K$ to be the inclusion $h^{4+\alpha}_{0,e}(\mathbb{S})\hookrightarrow h^{1+\alpha}_{0,e}(\mathbb{S})$
we find all assumptions of Theorem \ref{T:C-R} fulfilled.
Moreover, in virtue of $\lambda'(\overline\gamma_1)<0,$ and since $\gamma_1'(\varepsilon)>0$ near $\varepsilon>0$ (respectively negative when $\varepsilon<0$) (cf. Theorem \ref{T:P3}), we conclude
 that the eigenvalue  $\mu(\varepsilon)$ of $\partial_f\Phi(\gamma_1(\varepsilon),f_1(\varepsilon))$ must have positive sign for $\varepsilon$ close to $0$.
 Hence, $f_1(\varepsilon)$ is an  unstable steady-state  of \eqref{eq:CPred} when $\gamma=\gamma_1(\varepsilon)$.
The proof is now completed.
\qed 
\end{proof}

\section{Appendix}
Let us now assume that both fluids have the same viscosity $\mu=\mu_\pm$.
We study the problem \eqref{eq:S}  in a  frame moving  with constant velocity $V>0,$ by using the results
already established in Theorem \ref{T:main} and Theorem \ref{T:main2}. 
More explicitly, in  a infinitely long vertical column consisting of two fluids we describe the motion of the interface separating the fluids and which is located between two parallel 
lines $\Gamma_\pm(t)$ which move  with constant velocity $V.$
We impose constant normal velocity $V$ on $\Gamma_+(t)$ and prescribe the pressure on $\Gamma_-(t).$   
This situation is similar to that studied in \cite{ST, SCH} where it is assumed that the velocity at $\pm \infty$ is equal to $(0,V).$

Consider now the global  solution $f$ of problem \eqref{eq:S} under constant boundary conditions $g_2=0$ and $g_1\equiv c\in\mathbb{R},$
with initial data $f_0\in h^{4+\alpha}_0(\mathbb{S}).$
We then define 
\begin{align*}
h(t)&:=f(t)+tV\in h^{4+\alpha}(\mathbb{S});\\[1ex]
\Omega_+(t,h)&:=\{(x,y)\,:\, h(t)<y<1+tV\};\\[1ex]
\Omega_-(t,h)&:=\{(x,y)\,:\, -1+tV<y<h(t)\};\\[1ex]
v_\pm(t,x,y)&:=u_{\pm}(x,y-tV)-\frac{\mu V}{k}y\pm\frac{\varpi V}{2}t, \qquad (x,y)\in\Omega_\pm(t,h).
\end{align*}
Setting $\Gamma_\pm(t):=\mathbb{S}\times\{\pm 1+tV\},$ we see that $(h,v_+,v_-) $ is the global solution of the Muskat problem in the moving frame:
\begin{equation}\label{eq:SV}
\left\{
\begin{array}{rllllll}
\Delta v_\pm&=&0&\text{in}& \Omega_\pm(t,h),\\[2ex]
\displaystyle -\frac{k}{\mu}\partial_y v_+&=&V&\text{on}& \Gamma_+(t),\\[2ex]
 v_-&=&\displaystyle C-\left(\frac{\mu V^2 }{k}+\frac{\varpi V}{2}\right) t&\text{on}&\Gamma_-{(t)},\\[2ex]
v_+-v_-&=&\gamma\kappa_{\Gamma(h)}+\varpi  h&\text{on}& \Gamma(h(t)),\\[2ex]
\displaystyle{\partial_th+\frac{k\sqrt{1+f'^2}}{\mu}\partial_\nu v_\pm}&=&0 &\text{on}& \Gamma(h(t)),\\[2ex]
h(0)&=&f_0, &
\end{array}
\right. 
\end{equation} 
with $t\in[0,\infty)$ and $C:=(c+\mu V)k.$  
Defining the velocity field $\vec v^V_\pm$ by Darcy's law 
$\vec v^V_\pm:=-(k/\mu)\nabla v_\pm$ in $\Omega_\pm(t,h),$
 the first boundary condition in \eqref{eq:SV} asserts that the normal velocity on $\Gamma_+(t)$ is constantly equal to $V$.
Complementary to \cite[Theorem 4.2]{SCH} (we treat here fluids with different densities and the same viscosity), our next result states not only global existence for small data but also 
exponentially fast convergence to a flat solution.
In virtue of Theorem \ref{T:main}, Theorem \ref{T:main2}, and Theorem \ref{T:P3} we obtain:
\begin{thm}\label{T:mFra} Let \eqref{eq:cond1} hold true and $C\in\mathbb{R}, V \geq0$ be fixed. 
Given $\omega\in (0,\tanh(1)(\gamma-\varpi)/(\mu(1+\tanh(1))))$, there exist positive constants $\delta, M$ such that for all $f_0\in h^{4+\alpha}_0(\mathbb{S})$ with $\|f_0\|_{h^{4+\alpha}(\mathbb{S})}\leq \delta,$
the unique solution to \eqref{eq:SV} exists in the large and 
\[
\|h(t)-tV\|_{h^{4+\alpha}(\mathbb{S})}+\|\partial_th(t)-V\|_{h^{1+\alpha}(\mathbb{S})}\leq Me^{-\omega t}\|f_0\|_{h^{4+\alpha}(\mathbb{S})}
\]
for all $t\geq 0.$
Moreover, if $\rho_+>\rho_-$ and $\gamma>0,$ there exist traveling finger shaped global solutions of \eqref{eq:SV}.
\end{thm}

\begin{rem}
Though the normal velocity on $\Gamma_-(t)$ is not constant, in the limit we have $\langle\vec v^V_-|\nu\rangle\to_{t\to\infty} V$  exponentially fast if \eqref{eq:cond1} holds true.  
If $\gamma=0,$ relation \eqref{eq:cond1}  ensures that the more dense fluid  lies beneath.
\end{rem}

\subsection*{Acknowledgement}
The authors wish to thank the anonymous referees for their constructive suggestions and comments which have improved the quality of the paper..


\begin{thebibliography}{99}


\bibitem{Ambr}
 Ambrose,~D.~M.,
 Well-posedness of two-phase Hele--Shaw flow without surface tension.
{\it European J. Appl. Math.} 15 (2004), 597 -- 607.  


\bibitem{Am} 
 Amann,~H.,
  {\it Linear and Quasilinear Parabolic Problems, Volume I}.  Basel: Birkh\"auser 1995.
  
 \bibitem{BT} 
 Buffoni,~B. and Toland,~J.,
{\it Analytic Theory of Global Bifurcation: An Introduction}.  New Jersey: Princeton 2003.  



   \bibitem{CC}   
Cengel,~Y.~A. and Cimbala,~J.~M.,
 {\it Fluid Mechanics: Fundamentals and Applications}.
    New York:  McGraw-Hill 2006.
    
 \bibitem{CCG1}
  C\'ordoba,~A.,  C\'ordoba,~D., and  Gancedo,~F.,
  Interface evolution: the Hele-Shaw and Muskat problems. 
  {\it Ann. Math.}, in press.  
 

  \bibitem{CCG2}
C\'ordoba,~A.,  C\'ordoba,~D., and  Gancedo,~F.,
  The Rayleigh-Taylor condition for the evolution
of irrotational fluid interfaces, 
{\it Proc. Natl. Acad. Sci. USA}   106   (2009)(27),  10955 -- 10959.  
  
 
\bibitem{CR}   
Crandall,~M.~G. and  Rabinowitz,~P.~H.,
 Bifurcation from simple eigenvalues.
    {\it J. Funct. Anal.} 8 (1971), 321 -- 340. 
    
  \bibitem{CR3}   
  Crandall,~M.~G. and  Rabinowitz,~P.~H.,
 Bifurcation, perturbation of simple eigenvalues,
and linearized stability.
  {\it  Arch. Rational Mech. Anal.}  52 (1973), 161 -- 180.
    
 \bibitem{EEM} 
 Ehrnstr\"om,~M.,    Escher.~J., and  Matioc,~B.--V.,
 Steady-state fingering patterns for a periodic Muskat problem. 
submitted.   

    


\bibitem{EM1} 
  Escher,~J. and   Matioc,~B.--V.,
 A moving boundary problem for periodic Stokesian Hele-Shaw flows.
 {\it Interfaces Free Bound.} 11 (2009), 119 -- 137.
 
 

\bibitem{EM5} 
 Escher,~J. and   Matioc,~B.--V.,
 Multidimensional Hele-Shaw flows modeling Stokesian fluids.
{\it Math. Methods Appl. Sci.} 32 (2009), 577 -- 593.




\bibitem{ES98} 
 Escher,~J. and  Simonett,~G.,
 A center manifold analysis for the Mullins-Sekerka model.
{\it J. Differential Equations} 143 (1998), 267 -- 292.



\bibitem{FT} 
 Friedman,~A. and Tao,~Y., 
 Nonlinear stability of the Muskat problem with capillary pressure at the free boundary.
{\it  Nonlinear Anal.} 53 (2003), 45 -- 80.




\bibitem{HLS} 
  Hou,~T.~J.,  Lowengrub,~J.~S., and  Shelley,~M.~J.,
 Removing the stiffness from interfacial flows with surface tension.  
{\it J. Comput. Phys.} 114 (1994), 312 -- 338.

  

 \bibitem{KPS} 
  Konic,~L., Shelley,~M.~J., and  Palffy-Muhoray,~P., 
 Models of non-Newtonian Hele-Shaw flow.
 {\it Phys. Rev. E}  54 (1996)(5), R4536--R4539.
 
 

\bibitem{La}
 Ladyzhenskaya,~ O.~A. and   Uraltseva,~N.~N.,
 {\it Linear and Quasilinear Elliptic Equations}.  New York: Academic Press 1968. 

 
 
\bibitem{L} 
 Lunardi,~A.,  
{\it Analytic Semigroups and Optimal Regularity in Parabolic Problems}.
   Basel: Birkh\"auser 1995.


\bibitem{HK}   
Kielh\"ofer,~H.,
 {\it Bifurcation Theory: An Introduction with Applications to PDEs}.  New York: Springer--Verlag 2004.
 

  
\bibitem{M}  
 Muskat,~M., 
Two fluid systems in porous media. The encroachment of water into an oil sand.
   {\it Physics} 5 (1934), 250 -- 264.   
 


\bibitem{ST}
 Saffman,~P.~G. and  Taylor,~G.~I., 
 The penetration of a fluid into a porous medium or Hele--Shaw cell containing a more viscous fluid.
 {\it Proc. R. Soc. A} 245 (1958), 312 -- 329. 
    

\bibitem{SCH}
  Siegel,~M.,  Caflisch,~R.~E., and  Howison,~S.,
 Global existence, singular solutions, and ill-Posedness for the Muskat problem.
 {\it Comm. Pure Appl. Math.}  57 (2004), 1374 -- 1411.



\bibitem{Y1} 
Yi,~F., 
 Local classical solution of Muskat free boundary problem.
   {\it J. Partial Diff. Eqs.} 9 (1996), 84 -- 96. 
    
 \bibitem{Yi2}   
 Yi,~F., 
 Global classical solution of Muskat free boundary problem.
{\it  J. Math. Anal. Appl.}   288 (2003), 442 -- 461.  
\end{thebibliography}
\end{document}